\documentclass[11pt]{amsart}

\usepackage[pdftex]{graphicx}
\usepackage[ddmmyyyy]{datetime}
\usepackage{amssymb,amsfonts,amsmath}
\usepackage{amsthm,xypic,enumerate}
\usepackage{color,tikz,pgf}
\usepackage{fancyhdr}
\usetikzlibrary{shapes}

\input xy
\xyoption{all}
\usepackage{hyperref}

\newcommand{\N}{\mathbb{N}}
\newcommand{\R}{\mathbb{R}}
\newcommand{\mms}{\mathcal{S}}
\newcommand{\mmS}{\mathcal{S}}
\renewcommand{\mms}{\mathcal{S}}
\newcommand{\mmC}{\mathcal{C}}
\newcommand{\mmN}{\mathcal{N}}
\newcommand{\mmR}{\mathcal{R}}
\newcommand{\omR}{\overline{\R}}
\newcommand{\omN}{\overline{\N}}
\newcommand{\bn}{\mathbf{1}}
\newcommand{\on}{\mathcal{O}(\mmN)}

\newcommand{\mmY}{\mathcal{Y}}

\DeclareMathOperator*{\im}{im}
\define\ta{\tilde{a}}
\define\tb{\tilde{b}}

\newenvironment{enumerate*}[1][{}]{\begin{itemize}\renewcommand{\labelitemi}{#1}}{\end{itemize}}

\numberwithin{equation}{section}
\newtheorem{theorem}[equation]{Theorem}
\newtheorem{proposition}[equation]{Proposition}
\newtheorem{corollary}[equation]{Corollary}
\newtheorem{lemma}[equation]{Lemma}

\theoremstyle{definition}
\newtheorem{definition}[equation]{Definition}
\newtheorem{remark}[equation]{Remark}
\newtheorem{example}[equation]{Example}

\title[Preclusion of switch behavior in  reaction networks]{Preclusion of switch behavior in  reaction networks with mass-action kinetics}

\author{Elisenda Feliu, Carsten Wiuf}

\thanks{{\it Authors affiliation}: Bioinformatics Research Centre, Aarhus University, C. F. M\o llers All\'{e} 8, DK-8000 Aarhus, Denmark}
\thanks{{\it Corresponding author}: Elisenda Feliu, efeliu@birc.au.dk}

\date{\today}

\begin{document}

\maketitle

\begin{abstract}
We provide a Jacobian criterion that applies to arbitrary chemical reaction networks taken with mass-action kinetics to preclude the existence of multiple positive steady states within any stoichiometric class for any choice of rate constants.   
We are concerned with the characterization of injective networks, that is, networks for which  the species formation rate function is injective  in the interior of the positive orthant within each stoichiometric class. We show that a network is injective if and only if the determinant of the Jacobian of a certain function does not vanish. The function consists of   components of the species formation rate function and a maximal set of independent conservation laws.
The determinant of the function is a polynomial in the species concentrations and the rate constants (linear in the latter) and its coefficients are fully determined. 
The criterion also precludes  the existence of degenerate steady states. Further, we relate injectivity of a chemical reaction network to that of the chemical reaction network obtained by adding outflow, or degradation, reactions for all species.


\end{abstract}

\section{Introduction}
Multistationarity in cellular systems provides  a mechanism for switching between different cellular responses and can be crucial for cellular decision making.  Even though different features, such as feedback loops, are known that facilitate multistationarity in systems, it is in general difficult to decide whether a particular system has  the capacity to exhibit multiple steady states. Typical systems are high-dimensional and contain many  parameters that are unknown or poorly determined.  In order to determine the steady states of such a system,  the simultaneous solutions to a large set of equations taken together with the unknown parameters is required. In general, this is an impractical task. Various criteria have therefore been developed to preclude the existence of multiple (positive) steady states.  These criteria typically utilize the structure or qualitative features of the system \cite{feinbergnotes,craciun-feinbergII,banaji-craciun1} or properties of the class of kinetics that are allowed \cite{craciun-feinbergI,craciun-feinberg-semiopen,Banaji-donnell}. 

It is the aim of this paper to introduce a criterion for a chemical reaction network taken with mass-action kinetics to preclude the existence of multiple positive steady states within any stoichiometry class for any choice of rate constants. The criterion is based on the  species formation rate function and characterizes when this function is injective  for positive concentration vectors within each stoichiometric class. If this is the case then the network is said to be injective (Definition 5.1) and there cannot exist  multiple positive steady states within a stoichiometric class.

We provide a Jacobian criterion that characterizes injectivity for any network.  The criteria is computationally tractable and  extends  the Jacobian criterion for \emph{fully open} networks in \cite{craciun-feinbergI}. In a fully open network all  chemical species are assumed to flow out of the system or, alternatively,  all  species are being  degraded. If the dimension of the stoichiometric space is less than the number of species then
 the Jacobian of the species formation rate function is always singular and the criterion for fully open networks does not apply.
We replace the species rate formation function by  a new function (Definition 4.4) obtained by replacing some components of the species formation rate function by equations for independent conservation laws.  We show that a network is injective if and only if the determinant of the Jacobian of the modified function does not vanish (Corollary 5.4). If this is the case then any positive steady state is \emph{non-degenerate} (Corollary 5.5).
The criterion depends only on the structure of the network and, therefore, is easy to implement using any computational algebra software.

Any network can be seen as a subnetwork of a fully open network by adding outflow reactions. 
We show that the determinant of the Jacobian related to  the original network can be recovered from that of the associated  fully open network. Further, we show that if the fully open network is injective then either the original network is injective as well and all positive steady states are non-degenerate, or all steady states of the original network are degenerate (Theorem 9.1).  Further, the cases for which the latter occurs are characterized (Corollary 8.1). 

Our work builds on  previous  work by Craciun and Feinberg \cite{craciun-feinbergI} on injectivity of networks in the context of a \emph{continuous flow stirred tank reactor} (here called fully open networks). 
In two subsequent papers  Craciun and Feinberg relate these results  to arbitrary networks.
First, in \cite{craciun-feinberg} they show that if a fully open network    does not have the capacity to admit multiple positive steady states, then any network resulting from removing outflow reactions cannot have multiple non-degenerate positive steady states within any stoichiometric class.
Later, in \cite{craciun-feinberg-semiopen}, they provide a (sufficient) condition to ensure that in an arbitrary network degenerate steady states cannot exist assuming that the fully open network is injective. 
Therefore, the combination of the two results gives a criterion to preclude multiple  positive  steady states. It consists of first deciding whether the associated fully open network is injective and then determining if degenerate steady states can occur.
However, a network can be injective even if the associated fully open network is not.

 The work presented here  
 provides a direct path to preclude multiple steady states by avoiding the detour to fully open networks. We show that injectivity of an arbitrary network can be assessed regardless of the injectivity of the associated fully open network.  Further,  if the associated fully open network is injective, the occurrence of degenerate steady states in a  network is completely characterized.

A different route to injectivity of a fully open network was taken by Banaji et al. in \cite{Banaji-donnell}. A criterion is given that ensures that minus  the Jacobian of the species formation rate function is a $P$-matrix (the definition is given in \S\ref{sec10}). It then follows from the results of \cite{Gale:1965p474} that the network is injective. Our results imply that, after changing the sign of certain rows, the Jacobian of the modified species formation rate function is a $P$-matrix and it follows that the network is injective as well. Using this approach and the notion of strongly sign-determined matrices, Banaji et al.  extend in \cite{Banaji-donnell}  the injectivity results of Craciun and Feinberg for fully open networks taken with mass-action kinetics to kinetics satisfying some mild conditions (see also \cite{banaji-craciun1}). 
Our work is currently restricted to  mass-action kinetics  and the extension to general kinetics is currently being investigated.

The outline of the paper is as follows. In \S\ref{CRN} we introduce some notation and the main definitions relating to networks  and  mass-action kinetics. In \S\ref{sec3} we introduce the stoichiometric classes and the distinction between fully open and closed networks. We proceed in \S\ref{sec:degenerate} to study degenerate steady states. Injectivity of networks is discussed in \S\ref{sec5}, where the definition of injectivity and the Jacobian criterion are introduced. In \S\ref{sec6} and \S\ref{closed} we focus on open and closed networks, respectively. Section~\ref{sec8} provides a characterization of networks with only degenerate steady states. Finally, in \S\ref{sec9} we relate injectivity of open networks to that of closed networks  and in \S\ref{sec10} the relationship between  $P$-matrices and injectivity is discussed. We end with a few remarks including  a summary  (Figure \ref{summary}) of how our work relates to the previous work of Craciun and Feinberg.

\section{Chemical reaction networks with mass-action kinetics}\label{CRN}

\subsection{Notation}
Let $\R_{+}$ denote the set of positive real numbers (without zero) and $\omR_{+}$  the set of non-negative real numbers (with zero). Similarly, let $\omN$ be the set of non-negative integers. Given a finite set $\mathcal{E}$, let $\omN^{\mathcal{E}}$ be the semi-ring of formal sums $v=\sum_{E\in \mathcal{E}} \lambda_E E$, with $\lambda_E\in \omN$. If $\lambda_E\in \N$   for all $E\in \mathcal{E}$,  then we write $v\in \N^{\mathcal{E}}$. The semi-rings $\omR^{\mathcal{E}}_+$ and $\R^{\mathcal{E}}_+$ are defined analogously.

The ring of polynomials in $\mathcal{E}$ is denoted $\R[\mathcal{E}]$. 
The total degree of a monomial $\prod_{E\in \mathcal{E}} E^{n_E}$, with $n_E$ a non-negative integer for all $E$, is the sum of the degrees of  the  variables, $\sum_{E\in \mathcal{E}} n_E$. The degree of a polynomial is   the maximum of the total degrees of its monomials. 

If a polynomial $p$ vanishes for all assignments  $a\colon\mathcal{E}\rightarrow \R_+$ then $p=0$ identically. Further, if $p$ is a non-zero polynomial  in $\R[\mathcal{E}]$ such that the degree of each variable in each monomial is either $1$ or zero, then all the coefficients of $p$ are  non-negative if and only if $p(a(\mathcal{E}))>0$ for any assignment  $a\colon\mathcal{E}\rightarrow \R_+$.

\subsection{Chemical reaction networks}
Here we introduce  the definition of a chemical reaction network and some related concepts. See for instance \cite{feinbergnotes,Feinbergss} for background and extended discussions.  

\begin{definition}\label{crn}
A \emph{chemical reaction network} (or simply \emph{network})
consists of three finite sets:
\begin{enumerate}[(1)]
\item A set $\mmS$ of \emph{species}.
\item A set $\mmC\subset \omN^{\mmS}$ of \emph{complexes}.
\item A set  $\mmR\subset \mmC\times \mmC$ of \emph{reactions}, such that $(y,y)\notin \mmR$ for all $y\in \mmC$, and if $y\in \mmC$, then there exists $y'\in \mmC$ such that either $(y,y')\in \mmR$ or $(y',y)\in \mmR$. 
\end{enumerate}
\end{definition}
Following the usual convention, an element $r=(y,y')\in \mmR$ is  denoted  by $r\colon y\rightarrow y'$.
The \emph{reactant} and \emph{product} (complexes) of a reaction $r\colon y\rightarrow y'$  are $y$ and $y'$, respectively. By definition,   any complex is either the reactant or product  of some reaction.  
The zero complex $0\in \mmC$ is allowed by definition. Reactions of the form  $y\rightarrow0$ are called \emph{outflow} reactions and reactions of the form $0\rightarrow  y$ are called \emph{inflow} reactions  \cite{feinberg-horn-open}.  In particular, \emph{species inflow} and \emph{species outflow} reactions are reactions of the form $0\rightarrow S$ and $S\rightarrow 0$, respectively, for some $S\in \mmS$.
 
Let $n$ be the cardinality of $\mmS$. We fix an order in $\mmS$ so that $\mmS=\{S_1,\dots,S_n\}$ and  identify $\omN^{\mmS}$  with $\omN^n$. The species $S_i$ is identified with the $i$-th canonical $n$-tuple of $\omN^n$ with $1$ in the $i$-th position and zeroes elsewhere. Accordingly, a complex $y\in \mmC$ is given as $y=\sum_{i=1}^n y_i S_i$ or  $(y_1,\dots,y_n)$. Although $\omN^n$ is not a vector space, $n$-tuples $v\in \omN^n$ will be called \emph{vectors}.

\begin{example}\label{futile}
Consider the network with set of species $\mmS=\{S_1,S_2,S_3,S_4,S_5,S_6\}$, set of complexes $\{S_{1} + S_3,S_1+S_4,S_2+S_4,S_{2} + S_3,S_5,S_6 \}$ and reactions 

\vspace{0.1cm}
 \centerline{\xymatrix{
S_{1} + S_3 \ar@<0.3ex>[r]  & S_5  \ar@<0.3ex>[l]  \ar[r] & S_{1} + S_4  &
S_{2} + S_4 \ar@<0.3ex>[r]  & S_6  \ar@<0.3ex>[l]  \ar[r]  & S_{2} + S_3.
}}

\vspace{0.1cm}
\noindent  That is, the reactions are $r_1\colon  S_1+S_3\rightarrow S_5$, $r_2\colon  S_5\rightarrow S_1+S_3$, $r_3\colon S_5\rightarrow S_1+S_4$, $r_4\colon  S_2+S_4\rightarrow S_6$, $r_5\colon  S_6\rightarrow S_2+S_4$ and $r_6\colon S_6\rightarrow S_2+S_3$. This network is a main building block in protein modification systems and is  known as the \emph{futile cycle}. It assumes  the Michaelis-Menten enzyme mechanism in which a substrate $S_3$ is modified into a substrate $S_4$ through the  formation of an intermediate complex $S_5$. The reaction is catalyzed by an enzyme $S_1$. The modification can be reversed via a similar set of reactions with an intermediate complex $S_6$ and an enzyme $S_2$.  Each \emph{reversible} reaction is written as two \emph{irreversible} reactions, e.g.~the reactions $r_1$ and $r_2$ are considered  two distinct reactions and not one reversible reaction.
\end{example}

\subsection{Mass-action kinetics}
Let $\mathcal{N}=(\mmS,\mmC,\mmR)$ be a  network. We denote the molar concentration  of species $S_i$ at time $t$ by $c_i=c_i(t)$ and associate with any complex $y=(y_1,\dots,y_n)$ the monomial $c^y=\prod_{i=1}^n  c_i^{y_i}\in \R[c_1,\dots,c_n]$. For example, if $y=(2,1,0,1)\in \omN^4$, then the associated monomial is $c^y=c_1^2c_2c_4$. 

A  \emph{rate vector} is an element $\kappa=(k_{y\rightarrow y'})_{y\rightarrow y'} \in \R_+^{\mmR}$ given by the assignment of a positive  \emph{rate constant} $k_{y\rightarrow y'}\in \R_+$ to each reaction $r\colon y\rightarrow y'\in \mmR$. The (mass-action) \emph{species formation rate function} corresponding to the rate vector $\kappa$ is defined by:
\begin{eqnarray*}\label{massact}
\R^{n} & \xrightarrow{f_{\kappa}} & \R^{n}, \qquad 
 c  \mapsto  \sum_{y\rightarrow y'\in \mmR} k_{y\rightarrow y'} c^y (y' -   y).
\end{eqnarray*} 
Let $f_{\kappa,i}(c)$ denote the $i$-th entry of  $f_{\kappa}(c)$, that is
$f_{\kappa,i}(c)= \sum_{y\rightarrow y'\in \mmR} k_{y\rightarrow y'} c^y (y'_i -   y_i)$.

The set of reactions together with a rate vector give rise to a polynomial system of ordinary differential equations (ODEs): 
\begin{align*} 
\dot{c} &=f_{\kappa}(c).
\end{align*}
These ODEs describe the dynamics of the concentrations $c_i$ in time.
The  \emph{steady states} of the network with rate vector $\kappa$ are the solutions to the system of polynomial  equations in $c_1,\dots,c_n$ obtained by setting the derivatives of the  concentrations to zero:
\begin{align*} 
0 =& f_{\kappa}(c). 
\end{align*}
This system of equations is referred to as the steady-state equations. We are interested in the  positive steady states, that is, the  solutions $c$ to the steady-state equations such that all concentrations are  positive, $c\in \R^n_+$. 

\begin{example}
The  ODEs system of the futile cycle taken with mass-action kinetics is:
\begin{align*}
\dot{c}_{1} &=  - k_1 c_{1}c_3+ (k_2+k_3)c_5   & \dot{c}_4 &=-k_4 c_{2}c_4 +k_3c_5+k_6c_6   \\  \dot{c}_{2} &= - k_4 c_{2}c_4 +(k_5+k_6)c_6  
& \dot{c_5} &=  k_1  c_{1}c_3-(k_2+k_3)c_5        \\ 
\dot{c}_3 &= - k_1c_{1}c_3+ k_2c_5+k_6c_6  & \dot{c_6} &=  k_4 c_2c_4-(k_5+ k_6)c_6
\end{align*}
where   the rate constant of reaction $r_i$ is denoted by $k_i$.
\end{example}

\begin{remark} If $\kappa\in \R^{\mmR}_+$ and/or $c\in \R^n_+$ are not fixed then the function $f_{\kappa}(c)$ can be seen as a polynomial function taking values  in  $\R[c]$, $\R[\kappa]$ or $\R[c\cup \kappa]$. 
\end{remark}

\section{Stoichiometrically compatible steady states}\label{sec3}
The dynamics of a network might preserve quantities that remain constant  over time. If this is the case, the dynamics  takes place in a proper invariant subspace of $\R^{n}$, fixed by the initial concentrations $c_i(0)$ of the system.  Let $v\cdot v'$ denote the Euclidian scalar product of two vectors $v,v'$. Let $v^t$ denote the transpose of a vector $v$.

\begin{definition}
The \emph{stoichiometric subspace}  of a network $\mmN=(\mmS,\mmC,\mmR)$ is the following vector subspace of $\R^n$: 
$$\Gamma = \langle y'-y|\, y\rightarrow y' \in \mmR\rangle.   $$
Two  vectors   $c,c'\in \R^n$   are called \emph{stoichiometrically compatible} if $c-c'\in \Gamma$, or equivalently,  if $\omega\cdot c = \omega \cdot c'$  for all $\omega\in \Gamma^{\perp}$. 
\end{definition}

We denote by $s$ the dimension of $\Gamma$. Note that $\Gamma$ is independent of the choice of rate constants and depends only on the structure of the reactions. Being stoichiometrically compatible is an equivalence relation which partitions $\overline{\R}^n_+$ into classes, called \emph{stoichiometric classes}. In particular, 
the \emph{stoichiometric class} of a concentration vector $c\in \omR^n_{+}$ is $\{c+\Gamma\}\cap\omR_{+}^n$. 
 
 For any rate vector $\kappa\in \R^{\mmR}_+$, the image of $f_{\kappa}$ is contained in $\Gamma$. Thus, for any choice of rate vector $\kappa$, $c\in \R^n$, and $\omega\in \Gamma^{\perp}$, we have that $\omega\cdot f_{\kappa}(c)=0$  and thus  $\omega\cdot \dot{c}=0$.
If $\omega=(\lambda_1,\dots,\lambda_n)$ then  $\sum_{i=1}^n \lambda_i c_i$ is independent of time and  determined by the initial concentrations of the system. These conserved quantities are generally referred to as \emph{total amounts}. Since  $a-b\in \Gamma$ if and only if $\omega\cdot a = \omega \cdot b$  for all $\omega\in \Gamma^{\perp}$,  total amounts are conserved within each stoichiometric class and characterize the class.

\begin{example}
The stoichiometric space of the futile cycle (Example~\ref{futile}) is
\begin{align*}
\Gamma & =\langle  S_1+S_3-S_5, S_5-S_1-S_4,  S_2+S_4- S_6,S_6-S_2-S_3 \rangle \\ & = \langle  S_1+S_3-S_5, S_5-S_1-S_4,  S_2+S_4- S_6  \rangle
\end{align*}
and has dimension $s=3$. The dimension of the orthogonal space $\Gamma^{\perp}$ is $d=n-s=3$ as well and a basis is
$\Gamma^{\perp} =\langle  \omega^1,\omega^2,\omega^3\rangle$ with 
\begin{align}\label{perpbasis}
\omega^1 &= S_1+S_5, & \omega^2 &= S_2+S_6, & \omega^3 &= S_3+S_4+S_5+S_6.
\end{align}
Indeed, we have $\dot{c}_1+\dot{c}_5=\dot{c}_2+\dot{c}_6=\dot{c}_3+\dot{c}_4+\dot{c}_5+\dot{c}_6=0$.  In this example, three total amounts determine each stoichiometric class.
\end{example}

\begin{remark} Questions like ``How many steady states does a system possess?'' refer to the number of steady states  within each stoichiometric class.  If this restriction is not imposed and $s=\dim(\Gamma)<n$, then the steady states describe an algebraic variety of dimension at least one over the complex numbers.
\end{remark}

\begin{definition}
A network $\mmN=(\mmS,\mmC,\mmR)$ has the \emph{capacity for multiple positive steady states} if there exists a rate vector $\kappa\in \R_+^{\mmR}$ and distinct vectors  $a,b\in \R^n_+$ such that $a-b\in \Gamma$ and  $f_{\kappa}(a)=f_{\kappa}(b)=0$.
\end{definition}

Note that in this work we focus mainly on the existence of multiple \emph{positive} steady states. However, as we will show, the methods developed can preclude the existence of a specific type of  multiple steady states on the boundary of $\R_+^n$.

If a basis $\{\omega^1,\dots,\omega^d\}$ of $\Gamma^{\perp}$ is chosen, then a network $\mmN$  has the capacity for multiple positive steady states if there exists a rate vector $\kappa\in \R_+^{\mmR}$ and distinct  $a,b\in \R^n_+$ such that $f_{\kappa}(a)=f_{\kappa}(b)=0$ and $\omega^i\cdot a = \omega^i\cdot b$ for all $i=1,\dots,d$. In particular, if  the map $\overline{f}_{\kappa}:\R^n\rightarrow \R^{d+n}$ defined by $\overline{f}_{\kappa}(c)=(\omega^1\cdot c,\dots,\omega^d\cdot c,f_{\kappa,1}(c),\dots,f_{\kappa,n}(c))$ is injective, then the network does not have the capacity for multiple positive steady states. This function is the focus of study in this paper.

If  species in- or outflow reactions  exist  for all species in  a network then the stoichiometric space has maximal dimension $n$. Therefore,  the requirement $a-b\in \Gamma$ in the previous definition  is superfluous. 
Specifically, if there is a reaction $S\rightarrow 0$ or $0\rightarrow S$ for some species $S$ then there are no vectors in $\Gamma^{\perp}$ with non-zero $i$-th entry.

\begin{lemma}\label{nocons} Let $\mathcal{N}=(\mmS,\mmC,\mmR)$ be a network.
 If   $S_i\rightarrow 0\in \mmR$ or $0\rightarrow S_i\in \mmR$ for some $S_i\in \mmS$, then $\lambda_i=0$ for all $\omega=(\lambda_1,\dots,\lambda_n) \in \Gamma^{\perp}$. 
\end{lemma}
\begin{proof}
Since $\pm S_i\in \Gamma$, we have $0=\omega\cdot (\pm S_i)= \pm \lambda_i$.
\end{proof}

\begin{definition}\label{openclose}
Let $\mathcal{N}=(\mmS,\mmC,\mmR)$ be a network. We say that
\begin{enumerate}[(i)]
\item  $\mathcal{N}$ is \emph{open}  if $\Gamma= \R^n$.
\item $\mathcal{N}$ is \emph{fully open} if the outflow reaction $S_i\rightarrow 0$ belongs to $\mmR$ for all $S_i\in \mmS$.
\item $\mathcal{N}$ is \emph{closed} if $\Gamma\neq \R^n$.
\end{enumerate}
If $\mmN$ is any  network then the \emph{associated fully open network}, $\mmN^o=(\mmS,\mmC^o,\mmR^o)$, is the network with   $\mmC^o=\mmC\cup\mmS\cup \{0\}$ and $\mmR^o=\mmR\cup \{S_i\rightarrow 0|\, i=1,\dots,n\}$.
\end{definition}

Lemma \ref{nocons} ensures that a fully open network is also open. Fully open networks are  considered by Craciun and Feinberg in \cite{craciun-feinbergI} in the context of continuous  flow stirred tank reactors and  their results extend to arbitrary open networks. A closed network is allowed to have outflow reactions as well, but not for all species since $\Gamma\neq \R^n$. 

\begin{remark} A closed network could also be made open by adding species inflow reactions to the set of reactions, but to be fully open requires species outflow reactions.
\end{remark}

\medskip
\textbf{Notation.}
Let $\mathcal{N}=(\mmS,\mmC,\mmR)$ be a network and $\kappa$ a rate vector. At this point we have defined the species formation rate function $f_{\kappa}$, the stoichiometric space $\Gamma$, and used $n$ for the number of species, $s$ for the dimension of $\Gamma$ and $d=n-s$ for the dimension of $\Gamma^{\perp}$. 
None of these objects incorporate reference to the specific network in the notation. This notation is used without further mentioning  throughout the paper.
Additionally, to ease notation, in some examples species are called $A,B,\dots$ and the respective concentrations $c_A,c_B,\dots$.

\section{Degenerate steady states}\label{sec:degenerate}

For any function $f=(f_1,\dots,f_p)\colon\R^m\rightarrow \R^p$ let $J_c(f)$ denote the Jacobian of $f$ at $c$, that is, the $p\times m$ matrix  with entry $(i,j)$ being $\partial f_i(c)/\partial c_j$.

\begin{definition}\label{def:deg} Let $\mathcal{N}=(\mmS,\mmC,\mmR)$  be a network and   $\kappa\in \R_+^{\mmR}$ a rate vector. A steady state $c^*\in \R^n$ of $\mathcal{N}$  is \emph{degenerate} if $\ker(J_{c^*}(f_{\kappa}))\cap \Gamma \neq \{0\}$.
\end{definition}

It is proven in  \cite[$\S$5]{Feinbergss} that for any $\kappa\in \R_+^{\mmR}$, $c\in \R_+^n$, and $\gamma\in \R^n$, 
\begin{equation}\label{jacform}
J_c(f_{\kappa})(\gamma)  = \sum_{y\rightarrow y'\in \mmR} k_{y\rightarrow y'} c^y(y*_c \gamma) (y'-y),\quad \textrm{where}\quad v*_cw=\sum_{i=1}^n\frac{v_iw_i}{c_i}.
\end{equation}

It is our aim to understand $\ker(J_{c^*}(f_{\kappa}))\cap \Gamma$. We find a criterion to determine whether this intersection consists of the zero vector or is a proper subspace. The criterion does not require the computation of $\ker(J_{c^*}(f_{\kappa}))$. In particular, we find that the existence of degenerate steady states is   linked to the function $\overline{f}_{\kappa}$ being injective. 

 Since vectors of $\Gamma$ are characterized by being orthogonal to all vectors in $\Gamma^{\perp}$, we obtain the following proposition (stated here for a general vector subspace $F$).

\begin{proposition}\label{degenerate} Let $\mathcal{N}$ be a network and $\kappa\in \R_+^{\mmR}$ a rate vector.   Let $F$ be a vector subspace of $\R^n$, $\{\omega^1,\dots,\omega^m\}$ a basis of $F^{\perp}$ and  define $\overline{f}_{\kappa}\colon \R^n\rightarrow \R^{m+n}$ by $$\overline{f}_{\kappa}(c)=(\omega^1\cdot c,\ldots, \omega^m\cdot c,f_{\kappa,1}(c),\dots,f_{\kappa,n}(c)).$$ Fix $c^*\in \R^n$. Then, 
$\ker(J_{c^*}(f_{\kappa}))\cap F= \{0\}$  if and only if the Jacobian $J_{c^*}(\overline{f}_{\kappa})$ of  $\overline{f}_{\kappa}$ at $c^*$  has maximal rank $n$.
\end{proposition}
\begin{proof}
Let $J_{c^*}(\overline{f}_{\kappa})$ be the Jacobian of $\overline{f}_{\kappa}(c)$ at $c^*$. It is an $(m+n)\times n$ matrix. The rank of $J_{c^*}(\overline{f}_{\kappa})$ is maximal if and only if $\ker(J_{c^*}(\overline{f}_{\kappa}))= \{0\}$.
For $i=1,\dots,m$, the $i$-th row of  $J_{c^*}(\overline{f}_{\kappa})$  equals the vector $\omega^i$.  The lower $n\times n$ matrix of $J_{c^*}(\overline{f}_{\kappa})$ (obtained by removing the first $m$ rows) is equal to $J_{c^*}(f_{\kappa})$. A vector $v$ belongs to $F$ if and only if $\omega^i\cdot v=0$ for all $i$. It follows that $v\in \ker(J_{c^*}(\overline{f}_{\kappa}))$ if and only if $v\in \ker(J_{c^*}(f_{\kappa}))$ and $v\in F$. Thus, $\ker(J_{c^*}(f_{\kappa}))\cap F= \{0\}$ if and only if $\ker(J_{c^*}(\overline{f}_{\kappa}))=\{0\}$. \end{proof}

By  letting $F=\Gamma$ in the previous lemma,  we have shown that a steady state $c^*$ is non-degenerate if and only if $J_{c^*}(\overline{f}_{\kappa})$ has maximal rank $n$.  Since $J_{c^*}(\overline{f}_{\kappa})$ is a $(n+d)\times n$ matrix, $d$ rows of the matrix are linearly dependent of the remaining $n$. We describe now a procedure that is independent of $c$ and $\kappa$ to determine $d$ rows with this property. Thus, it can be used to determine the rank of $J_{c^*}(\overline{f}_{\kappa})$.

\begin{lemma}\label{omegajac} If $\omega\in \Gamma^{\perp}$ then
 $\omega^t J_c(f_{\kappa})=0$ for all $c\in \R^n$ and all rate vectors $\kappa\in \R^{\mmR}_+$.
\end{lemma}
\begin{proof}
If $\omega\in \Gamma^{\perp}$ then $\omega\cdot f_{\kappa}(c)=0$ for all $c$ and hence the scalar product vanishes as a polynomial in $c$. It follows that 
$0  =\frac{\partial(\omega\cdot f_{\kappa}(c))}{\partial c_i} = \omega\cdot \frac{\partial f_{\kappa}(c)}{\partial c_i}$ for all $i=1,\dots,n$ and $\omega^t J_c(f_{\kappa})=0$.
\end{proof}

Hence, each $\omega\in \Gamma^{\perp}$ provides a vanishing linear combination of the rows of $J_c(f_{\kappa})$.

\begin{definition}\label{extended}
A basis $\{\omega^1,\dots,\omega^d\}$  of  $\Gamma^{\perp}$ with $\omega^i=(\lambda_1^i,\dots,\lambda_n^i)$ is said to be \emph{reduced} if 
$\lambda_i^i=1$ for all $i$ and $\lambda^i_j=0$ for all $j=1,\dots,\widehat{j},\dots,d$. Given a rate vector $\kappa$, the \emph{associated extended rate function} $\widetilde{f}_{\kappa}:\R^n\rightarrow \R^n$ is the function defined by 
$$\widetilde{f}_{\kappa}(c)=(\omega^1\cdot c,\dots,\omega^d\cdot c,f_{\kappa,d+1}(c),\dots,f_{\kappa,n}(c)).$$
\end{definition}
After reordering  the species, such a basis of  $\Gamma^{\perp}$ always exist (use for instance Gaussian elimination on any basis of $\Gamma^{\perp}$). We assume from now on that a reduced basis is chosen, implying that the species are ordered accordingly. 

\begin{example} The basis $\omega^1,\omega^2,\omega^3$ of $\Gamma^{\perp}$ of the futile cycle given in \eqref{perpbasis} is reduced. The associated extended rate function is 
\begin{multline*}
\widetilde{f}_{\kappa}(c)=(c_1+c_5,c_2+c_6,c_3+c_4+c_5+c_6, -k_4 c_{2}c_4 +k_3c_5+k_6c_6, \\ k_1  c_{1}c_3-(k_2+k_3)c_5, k_4 c_2c_4-(k_5+ k_6)c_6). 
\end{multline*}
\end{example}

Let $J_{c,i}(f_{\kappa})$ denote the $i$-th row of  $J_{c}(f_{\kappa})$. If $\{\omega^1,\dots,\omega^d\}$ is a reduced basis with $\omega^i=(\lambda_1^i,\dots,\lambda_n^i)$, then it follows from Lemma \ref{omegajac}  that  $$0=(\omega^i)^t J_c(f_{\kappa})=  J_{c,i}(f_{\kappa})+ \sum_{j=d+1}^n  \lambda_j^iJ_{c,j}(f_{\kappa}),\qquad i=1,\dots,d.$$
Thus, the rows $1,\dots,d$ of   $J_c(f_{\kappa})$ are linear combinations of the rows $d+1,\dots,n$. It follows that the rank of $J_{c^*}(\overline{f}_{\kappa})$ equals the rank of $J_{c^*}(\widetilde{f}_{\kappa})$.  The latter  is a square $n\times n$ matrix and has maximal rank if and only if its determinant does not vanish.

\begin{corollary}\label{jaccomp} Let $\mathcal{N}$ be a network, $\kappa\in \R_+^{\mmR}$ a rate vector, 
 $\{\omega^1,\dots,\omega^d\}$ a reduced basis of  $\Gamma^{\perp}$ and $\widetilde{f}_{\kappa}(c)$  the associated extended rate function. If $c^*\in\R^n$, then $\ker(J_{c^*}(f_{\kappa}))\cap \Gamma = \{0\}$ if and only if $\det(J_{c^*}(\widetilde{f}_{\kappa}))\neq 0$.  
 In particular,  a steady state $c^*\in \R^n$ of $\mmN$ is degenerate if and only if $\det(J_{c^*}(\widetilde{f}_{\kappa}))=0$. 
\end{corollary}

\begin{example}\label{futile:jac}
 The Jacobian matrix $J_{c}(\widetilde{f}_{\kappa})$ of the futile cycle associated with  
the reduced basis of $\Gamma^{\perp}$ in \eqref{perpbasis} is
$$
J_{c}(\widetilde{f}_{\kappa}) = \left(\begin{array}{cccccc}
1 & 0 & 0 & 0 & 1 & 0 \\ 0 & 1 & 0 & 0 & 0 & 1 \\  0 & 0 & 1 & 1 & 1 & 1 \\  0 &   -k_4c_4 & 0 &   -k_4 c_{2} & k_3 & k_6 \\ k_1  c_3 & 0 & k_1 c_{1} & 0 & -k_2-k_3 & 0 \\  0 & k_4c_4 & 0 &  k_4c_2 & 0 &  -k_5- k_6 
\end{array}\right).
$$
The determinant of $J_{c}(\widetilde{f}_{\kappa})$  is 
\begin{multline}\label{ex:jac}
 -\det(J_{c}(\widetilde{f}_{\kappa})) =(c_1c_2+c_1c_4)k_1k_3k_4 + c_1 k_1 k_3 k_5 \\ + (c_1c_2+c_2c_3) k_1 k_4 k_5   + c_2 k_2 k_4 k_5 + c_2 k_3 k_4 k_5 + c_1 k_1 k_3 k_6.
\end{multline}
All  coefficients of $ -\det(J_{c}(\widetilde{f}_{\kappa}))$ as a polynomial in $c,\kappa$ are positive. Thus $\det(J_{c}(\widetilde{f}_{\kappa}))$ does not vanish for any $c\in \R^n_+$ and $\kappa\in \R^{\mmR}_+$ and all positive steady states of the futile cycle are non-degenerate.
In this example, $\det(J_{c}(\widetilde{f}_{\kappa}))$ is linear in the rate constants.
 If at least one coefficient in the polynomial was negative then we could find $c,\kappa$ for which $ \det(J_{c}(\widetilde{f}_{\kappa}))=0$, implying that a degenerate steady state could occur.   This  observation holds for any network and will be discussed in the following sections.
\end{example}

\begin{remark}\label{kinetics-stoich} The minimal space containing the image of $f_{\kappa}$ is the  \emph{kinetic subspace}:
\begin{equation} 
\Lambda_{\kappa} := \langle \im f_{\kappa} \rangle \subseteq \Gamma. 
\end{equation}
In general, the two spaces $\Lambda_k$ and $\Gamma$ might not agree for a fixed rate vector $\kappa$. 
If this is the case then $\ker(J_c(f_{\kappa}))\cap\Gamma\neq \{0\}$ for any $c\in \omR^n_+$:
as above, if  $\omega\in \Lambda_{\kappa} ^{\perp}$ then   $0=\omega\cdot f_{\kappa}=\omega\cdot \dot{c}$. Similarly to the proof of  Lemma \ref{omegajac} we have that $\omega^t J_c(f_{\kappa})=0$ and hence $\langle \im J_c(f_{\kappa}) \rangle\subseteq \Lambda_{\kappa}$.  If $\Lambda_{\kappa}\subsetneq \Gamma$ then $\langle \im J_c(f_{\kappa}) \rangle$  has  at most dimension $s-1$ and it follows that the dimension of $\ker(J_c(f_{\kappa}))\cap\Gamma$ is at least $1$. Thus, if  the stoichiometric and the kinetic spaces do not agree for some $\kappa$, then all steady states corresponding to $f_{\kappa}$ are degenerate.
\end{remark}

\begin{example}
Consider  the network with reactions $r_1\colon A\rightarrow B$ and $r_2\colon A\rightarrow C$. The species formation rate function $f_{\kappa}\colon\R^3\rightarrow \R^3$ is 
$$f_{\kappa}(c_A,c_B,c_C)=(-(k_1+k_2)c_A, k_1c_A,k_2c_A), $$ 
where $k_i$ denotes the rate constant of reaction $r_i$. The stoichiometric space has dimension $2$ and a basis of $\Gamma^{\perp}$  is $A+B+C$. However, $k_2B-k_1C\in \Lambda_{\kappa}^{\perp}\setminus \Gamma^{\perp}$ and thus for all $\kappa$, all steady states are degenerate.
\end{example}

\begin{example}
The stoichiometric and the kinetic spaces of the previous example never agree. For some networks, the two spaces do not agree only for certain rate vectors. Consider for example the network with reactions  
$r_1\colon A\rightarrow B$, $r_2\colon A\rightarrow C$ and $r_3\colon B+C\rightarrow 2A$. 
The species formation rate function $f_{\kappa}:\R^3\rightarrow \R^3$ is 
$$f_{\kappa}(c_A,c_B,c_C)=(-(k_1+k_2)c_A+2k_3c_Bc_C, k_1c_A-k_3c_Bc_C,k_2c_A-k_3c_Bc_C). $$ 
If we let $\omega=(\lambda_1,\lambda_2,\lambda_3)$ then the equation $\omega\cdot f_{\kappa}(c)=0$ for all $c$ gives $2\lambda_1=\lambda_2+\lambda_3$ and $(\lambda_2-\lambda_3)(k_1-k_2)=0$. If $k_1\neq k_2$, then 
$(1,1,1)$ is a basis of $\Gamma^{\perp}$ as well as $\Lambda_{\kappa}^{\perp}$  and the stoichiometric and the kinetic spaces agree. However, if $k_1=k_2$ then   $(1,2,0) \in  \Lambda_{\kappa}^{\perp}\setminus \Gamma^{\perp}$ and the two spaces do not agree. Further, in this case all steady states are degenerate.
\end{example}

 \begin{remark} Note that  Proposition~\ref{degenerate}, Lemma~\ref{omegajac} and Corollary~\ref{jaccomp} do not depend on the kinetics being of mass-action type, but hold for general (differentiable) kine\-tics with $f_{\kappa}$ and $\tilde{f}_{\kappa}$ changed accordingly. In particular, Corollary~\ref{jaccomp} characterizes degenerate steady states in terms of the determinant of the Jacobian of $\tilde{f}_{\kappa}$. 
\end{remark}

\section{Injectivity of chemical reaction networks}\label{sec5}
Here we introduce  the notion of injectivity: a network is injective if for all rate vectors $\kappa$ the function $\widetilde{f}_{\kappa}$ is injective over $\R^n_+$.   The definition is an extension  of the definition of injectivity for fully open networks \cite[Def. 2.8]{craciun-feinbergI} to arbitrary networks and,  together with Proposition~\ref{inject} below, it is in line with the definition given in a recent paper on \emph{concordant} networks \cite{shinar-conc}. We  show that a network is injective if and only if the Jacobian of $\widetilde{f}_{\kappa}$ is non-singular for all positive concentrations $c$ and for all $\kappa$ .

In \cite[Def. 6.1]{craciun-feinberg-semiopen}  a definition of injectivity is given for networks with $\Gamma\subsetneq \R^n$. It relies on the Jacobian of the species formation rate function restricted to the stoichiometric space.  We prove in Theorem \ref{injecclose} below that our definition agrees with their definition. The equivalence is also  claimed   in Remark 6.4 in \cite{craciun-feinberg-semiopen} without a proof.

\begin{definition}\label{injective} A network $\mathcal{N}$ is said to be \emph{injective} if for any  distinct $a,b\in \R^n_+$ such that  $a-b\in \Gamma$, we have $f_{\kappa}(a)\neq f_{\kappa}(b)$.
\end{definition}

The definition of injectivity is restricted to the interior of the positive orthant. However, we  show below that being injective also precludes  the   existence of distinct $a,b$ such that $a-b\in \Gamma$ and $f_{\kappa}(a)= f_{\kappa}(b)$, provided either that  $a\in \omR^n_+$, $b\in \R^n_+$ or that  $a,b\in \omR^n_+$ are on the boundary of $\R_+^n$  and fulfill a certain condition.
 For $a\in\overline{\R}^n_+$, let  $I_a=\{i|a_i=0\}$ be the indices for which $a_i$ is zero and let  $\mmY_a=\{y|y\rightarrow y'\in\mmR,I_a\cap{\rm supp} (y)\not=\emptyset\}$, where ${\rm supp} (y)=\{i|y_i\not=0\}$ is the support of $y$. That is, $\mmY_a$ is the set of reactant complexes involving at least one species $S_i$ for which $a_i=0$. It follows that $y$ belongs to $\mmY_a$ if and only if $a^y=0$.

 \begin{proposition}\label{inject}
Let  $\mmN$ be a network. The following two statements are equivalent:
\begin{enumerate}
\item[(i)] $\mmN$ is injective.
\item[(ii)] For any distinct $a,b\in \overline{\R}^n_+$ such that  $\mmY_a\cap\mmY_b=\emptyset$ and $a-b\in \Gamma$, $f_{\kappa}(a)\neq f_{\kappa}(b)$.
\end{enumerate}
In particular, if $b\in\R^n_+$, that is $I_b=\emptyset$, then $\mmY_a\cap\mmY_b=\emptyset$ is fulfilled for any $a\in \omR^n_+$. 
\end{proposition}
\begin{proof}
(ii) obviously implies (i). To prove the reverse we assume that  there are $a,b\in\overline{\R}^n_+$, such that 
$\mmY_a\cap\mmY_b=\emptyset$,  $b-a\in\Gamma$ and $f_{\kappa}(a)=f_{\kappa}(b)$ and  show that $\mmN$ cannot be injective.  Let $\gamma_i=b_i-a_i$ and for a set $M\subseteq\{1,\ldots,n\}$ and $y\in\mmC$ define $x^{y_M}=\prod_{i\in M}x^{y_i}$. Further, define $I=I_a\cup I_b$ and $J=I^c$. If $I=\emptyset$, then clearly $\mmN$ cannot be injective. Hence, assume that  $I\not=\emptyset$. We seek to define a rate vector $\tilde{\kappa}=(\tilde{k}_{y\rightarrow y'})_{y\rightarrow y'}$ and  $\ta,\tilde{b}$ such that $\ta_i=a_i$, $\tilde{b}_i=b_i$ for $i\not\in I$,  $\ta_i,\tilde{b}_i>0$ and $\tilde{b}_i-\ta_i=\gamma_i$ for $i\in I$, and  
$$\sum_{y\rightarrow y'\in \mmR} \tilde{k}_{y\rightarrow y'} (\ta^y-\tilde{b}^y)(y'-y)=0,$$ 
that is, $f_{\tilde{\kappa}}(\ta)=f_{\tilde{\kappa}}(\tilde{b})$. Then, since $\ta,\tilde{b}\in \R^n_+$ and $\tb-\ta\in \Gamma$, it  follows that $\mmN$ is not injective.
The equality $f_{\kappa}(a)-f_{\kappa}(b)=0$ can be rewritten as:
{\small $$ 0=\sum_{\substack{y\rightarrow y' \\ y\in \mmY_a^c\cap\mmY_b^c}} k_{y\rightarrow y'} (a^y-b^y)(y'-y)+
\sum_{\substack{y\rightarrow y' \\ y\in \mmY_a\cap\mmY_b^c}} k_{y\rightarrow y'} (a^y-b^y)(y'-y)+
\sum_{\substack{y\rightarrow y' \\ y\in \mmY_a^c\cap\mmY_b}} k_{y\rightarrow y'} (a^y-b^y)(y'-y),$$}
Assume that we can find 
 $\ta_i,\tilde{b}_i>0$ with $\tilde{b}_i-\ta_i=\gamma_i$ for all $i$, $\tilde{a_i}=a_i$, $\tilde{b}_i=b_i$ for $i\not\in I$, and 
 $\alpha_y>0$ such that
 \begin{equation}\label{goal}
\frac{\alpha_y}{\alpha_y+1} \tilde{b}^y= \ta^y, \quad\textrm{if } y\in\mmY_a\cap\mmY_b^c,\quad\textrm{and} \quad \frac{\alpha_y}{\alpha_y+1} \tilde{a}^y= \tb^y,
 \quad\textrm{if } y\in\mmY_a^c\cap\mmY_b.
\end{equation}
Then, let $\tilde{k}_{y\rightarrow y'}=k_{y\rightarrow y'}$ for $y\in \mmY_a^c\cap\mmY_b^c$,  
$\tilde{k}_{y\rightarrow y'}=(\alpha_y+1)k_{y\rightarrow y'}b^y/\tilde{b}^y$ for  $y\in \mmY_a\cap\mmY_b^c$ and $\tilde{k}_{y\rightarrow y'}=(\alpha_y+1)k_{y\rightarrow y'}a^y/\tilde{a}^y$ for  $y\in \mmY_a^c\cap\mmY_b$.
With these definitions, if $y\in  \mmY_a\cap\mmY_b^c$ since $a^y=0$ we have
\begin{align*}
\tilde{k}_{y\rightarrow y'}(\ta^y-\tb^y) &=k_{y\rightarrow y'}  (\alpha_y+1) (\ta^y b^y/\tilde{b}^y - b^y) = 
k_{y\rightarrow y'} (\alpha_y b^y -  (\alpha_y+1)  b^y) \\ &  = k_{y\rightarrow y'} ( -b^y) =  k_{y\rightarrow y'} ( a^y-b^y)
\end{align*}
 Analogously, $ \tilde{k}_{y\rightarrow y'}(\ta^y-\tb^y) = k_{y\rightarrow y'} ( a^y-b^y)$  if $y\in  \mmY_a^c\cap\mmY_b$, using  $b^y=0$.
It follows that for $\tilde{\kappa}=(\tilde{k}_{y\rightarrow y'})_{y\rightarrow y'}$, we have $f_{\tilde{\kappa}}(\ta)-f_{\tilde{\kappa}}(\tilde{b}) = f_{\kappa}(a)-f_{\kappa}(b)=0$ as desired.

All that remains is to find  $\ta_i,\tilde{b}_i>0$ for $i\in I$ and 
 $\alpha_y>0$ such that \eqref{goal} is fulfilled.  
  If $y\notin \mmY_b$, then 
$\tilde{b}^y,\ta^y$ are independent of the value of $\ta_i,\tb_i$ for $i\in I_b$. Therefore, we define first $\ta_i,\tb_i$ for $i\in I_a$ and  focus on fulfilling \eqref{goal} for $y\in \mmY_a\cap\mmY_b^c$. Further, since we want $\tb_i-\ta_i=\gamma_i$, once $\ta_i$ is defined, we have $\tb_i = \ta_i + \gamma_i$.

Let $I_a=\{i_1,\ldots,i_m\}$,  $I_k=\{i_1,\ldots,i_{k}\}$ ($k\leq m$) with $I_0=\emptyset$ and $J_k=J\cup I_k$. Further, recursively define $\mmY_k=\{y \in \mmY_a |\ I_a\cap {\rm supp}(y)\subseteq I_k\} \setminus \mmY_{k-1}$  with $\mmY_0=\emptyset$. It follows that $\bigcup_{k=1}^m\mmY_k=\mmY$ such that any complex in $\mmY$ is in precisely one $\mmY_k$. That is, if $y\in\mmY_k$ then $i_k$ is the largest index in $I_a$ in the support of $y$.
We define $\ta_{i_k}, \tb_{i_k}$ recursively. Note that $\gamma_i>0$ for all $i\in I_a$. Assume that $\ta_{i_1},\ldots,\ta_{i_{k-1}}$ and $\tb_{i_1},\ldots,\tb_{i_{k-1}}$ are defined and hence the products $\tb^{y_{J_{k-1}}}$, $\ta^{y_{J_{k-1}}}$ are fixed.  For $i_k$ do the following: if $\mmY_k=\emptyset$, let $\ta_{i_k}=1$ and $\tb_{i_k}=1+\gamma_{i_k}$. If $\mmY_k\not=\emptyset$, observe that for  $y\in \mmY_k$ (for which $y_{i_k}\neq 0$), the equality \eqref{goal} can we written as 
\begin{equation}\label{goal2}
\frac{ \ta^{y_{i_k}}_{i_k} }{ (\ta_{i_k}+\gamma_{i_k})^{y_{i_k}} } = \frac{\alpha_y}{\alpha_y+1}\frac{ \tb^{y_{J_{k-1}}} }{ \ta^{y_{J_{k-1}}} },\quad\textrm{that is,}\quad \alpha_y = \frac{r_{i_k}(\ta_{i_k},y_{i_k})}{\frac{ \tb^{y_{J_{k-1}}} }{ \ta^{y_{J_{k-1}}} } - r_{i_k}(\ta_{i_k},y_{i_k})},
\end{equation}
with $r_{i_k}(\ta_{i_k},y_{i_k}):= \ta^{y_{i_k}}_{i_k}/(\ta_{i_k}+\gamma_{i_k})^{y_{i_k}}>0$. The function  $r_{i_k}$ is  increasing in $a_{i_k}$ and  $r_{i_k}(0,y_{i_k})=0$  for all $y_{i_k}$.
By defining $\ta_{i_k}>0$ arbitrarily such that $$r_{i_k}(\ta_{i_k},y_{i_k})<\min\Big( \frac{ \tilde{b}^{y_{J_{k-1}}}}{\ta^{y_{J_{k-1}}} }\Big|y \in\mmY_k \Big) $$  for all $y\in \mmY_k$ (which is a finite set), we obtain $\alpha_y>0$ fulfilling \eqref{goal2} as desired. 

The same procedure is applied to define $\ta_i,\tb_i$ for $i\in I_b$, with the roles of $a$ and $b$ reversed. In this case, however, $b_i=0$ implies that $\gamma_i<0$ and $\ta_i= \tb_i-\gamma_i$ for $i\in I_b$. Therefore, $r_{i_k}(\tb_{i_k},y_{i_k})$ becomes $\tb^{y_{i_k}}_{i_k}/(\tb_{i_k}-\gamma_{i_k})^{y_{i_k}} $, which also is increasing. 
\end{proof}

The assumption $\mmY_a\cap\mmY_b=\emptyset$  in Proposition~\ref{inject}(ii) cannot be relaxed. Consider the network with  $\mmS=\{A,B\}$, $\mmC=\{A+B,A+2B,0\}$ and reactions 
$A + B\rightarrow 0$, $A+2B \rightarrow 0$. This network is injective and open. If $a=(1,0)$ and $b=(0,1)$ then  $\mmY_a\cap\mmY_b =  \{A+B,A+2B\} \neq\emptyset$ and for any rate vector $\kappa$ we have $f_{\kappa}(a)=f_{\kappa}(b)$. 

\begin{remark} In \cite{shinar-conc}, a network $\mmN$ with \emph{arbitrary} kinetics is said to be injective if $f_{\kappa}(a)\not=f_{\kappa}(b)$ for any pair of stoichiometrically compatible concentration vectors $a,b$, at least one of which  is positive. The condition given in Proposition~\ref{inject}(ii) is slightly more general in that both $a$ and $b$ can be non-negative.
\end{remark}

If $\mathcal{N}$ is  injective, then for any choice of rate vector $\kappa$ at most one positive steady state can exist within each stoichiometric class, i.e., for every $c_0\in \R^n_+$ there exists at most one $c\in \{c_0+\Gamma\}\cap\R^n_+$ such that $f_{\kappa}(c)=0$. In other words, if $\mathcal{N}$ is  injective then $\mmN$ does not have the capacity for multiple positive steady states. However, the reverse might not be true: non-injective networks exist that do not have the capacity for multiple positive steady states.  An example is provided in Example~\ref{ex:noninj}.

The proof of the following theorem is adapted from the proof of \cite[Th 3.1]{craciun-feinbergI}.

\begin{theorem}\label{injecclose} Let $\mathcal{N}=(\mms,\mmC,\mmR)$ be a network. Then,
$\mmN$ is injective if and only if $\ker(J_{c}(f_{\kappa}))\cap \Gamma=  \{0\}$ for all $c\in \R^n_+$ and $\kappa\in \R^{\mmR}_+$.
\end{theorem}
\begin{proof}
 $\mmN$ is not injective if and only if there exists $\kappa\in \R^{\mmR}_+$ and distinct  $a,b\in \R^n_+$ such that   $a-b\in \Gamma$ and $f_{\kappa}(a)= f_{\kappa}(b)$.  Further, $\ker(J_{c}(f_{\eta}))\cap \Gamma\neq  \{0\}$ for some $c\in \R^n_+$ and $\eta\in \R^{\mmR}_+$ if and only if there exists $\gamma\in \Gamma$ such that $J_{c}(f_{\eta})(\gamma)=0$. 
By definition and  using \eqref{jacform}, 
\begin{eqnarray*}
f_{\kappa}(a)- f_{\kappa}(b) =0 & \Leftrightarrow &  \sum_{y\rightarrow y'\in \mmR} k_{y\rightarrow y'} (a^y-b^y)(y'-y)=0, \\
J_{c}(f_{\eta})(\gamma)=0  & \Leftrightarrow &  \sum_{y\rightarrow y'\in \mmR} \eta_{y\rightarrow y'} c^y(y*_c \gamma) (y'-y)=0.  \end{eqnarray*}
We will show that given distinct  $a,b\in \R^n_+$ such that  $\gamma:=a-b\in \Gamma$ and some rate vector $\kappa$, there exist  $c\in \R^n_+$ and $\eta\in \R^{\mmR}_+$ such that
\begin{equation}\label{eqrefs} k_{y\rightarrow y'} (a^y-b^y) = \eta_{y\rightarrow y'} c^y(y*_c \gamma)\end{equation}
and \emph{vice versa}. 
Consider distinct $a,b\in \R^n_+$ such that    $\gamma:=a-b\in \Gamma$, $\gamma\neq 0$. If $a_i-b_i\neq 0$ define $c_i:=\frac{a_i - b_i}{\log  a_i/b_i}>0$ and let $c_i=1$ otherwise.  Since $\gamma_i=0$ if $a_i=b_i$, we have 
$a^y =b^ye^{y*_c \gamma}$. Note that the signs of $a^y-b^y = b^y(e^{y*_c \gamma}-1)$ and $y*_c \gamma$ agree.  If $y*_c \gamma=0$, let  $\eta_{y\rightarrow y'}= 1$.
Otherwise, we let $\eta_{y\rightarrow y'}=\frac{ k_{y\rightarrow y'} (a^y-b^y)}{c^y(y*_c \gamma)} >0$ and  \eqref{eqrefs} is fulfilled.
Reciprocally, given $\gamma\in \Gamma$, $\gamma\neq 0$, $c\in \R^n_+$ and $\eta\in \R^{\mmR}_+$, define $a,b\in \R^n_+$ by 
$b_i=a_i=1$ if $\gamma_i=0$ and $b_i=\gamma_i/(e^{\gamma_i/c_i}-1)>0$, $a_i=b_ie^{\gamma_i/c_i}>0$ otherwise. Then  $a-b=\gamma\in \Gamma$ and  $a^y-b^y= b^y(e^{y*_c\gamma}-1)$ has the same sign as $y*_c \gamma$. If $a^y-b^y=0$, define $k_{y\rightarrow y'}=1$. Otherwise, define 
$k_{y\rightarrow y'} = \frac{\eta_{y\rightarrow y'} (y*_c \gamma)}{c^y (a^y-b^y)}>0$ 
and equality \eqref{eqrefs} is fulfilled.

\end{proof}

\begin{remark} The summand in $f_{\kappa}$ corresponding to an inflow reaction $0\rightarrow y'$ in $\mmR$ takes the form $k_{0\rightarrow y'} y'$ and thus is independent of the concentration vector $c$. It follows that $J_c(f_{\kappa})$ is independent of the presence or absence of inflow reactions in $\mmR$ and so is the property of being injective.
\end{remark}

We have thus obtained a characterization of injective networks in terms of the Jacobian associated with the species formation rate function. Together with Corollary \ref{jaccomp} we obtain:

\begin{corollary}\label{extendedcor} Let $\mathcal{N}$ be a network,
 $\{\omega^1,\dots,\omega^d\}$ a reduced basis of  $\Gamma^{\perp}$ and $\widetilde{f}_{\kappa}(c)$  the associated extended rate function. $\mmN$ is injective if and only if
$\det(J_{c}(\widetilde{f}_{\kappa}))\neq 0$ for all $\kappa\in \R^{\mmR}_+$ and $c\in \R^n_+$. 
\end{corollary}

\begin{remark} Statements similar to Theorem~\ref{injecclose} and Corollary~\ref{extendedcor}  cannot be stated for individual $\kappa$, since $\kappa$ and $\eta$  generally are different in the proof of  Theorem~\ref{injecclose}. However, if the total degree of each of polynomials in the components in $\tilde{f}_{\kappa}$ is at most two then $\tilde{f}_{\kappa}$ is an injective function if and only if the Jacobian is non-singular  \cite{Bass:1982p770}.
\end{remark}

Note that $\mmN$ is injective if and only if the extended rate function   $\widetilde{f}_{\kappa}$ associated 
with a reduced basis is injective over $\R^n_+$. Further, as is observed in  \eqref{ex:jac} for the futile cycle, $\det(J_{c}(\widetilde{f}_{\kappa}))$ is a homogeneous polynomial in the entries of $\kappa$ with total degree $s$, because  the  rows $1,\dots,d$ of $J_{c}(\widetilde{f}_{\kappa})$ are constants (that is, independent of $c$ and $\kappa$).

\begin{remark} A general version of the corollary above  has recently been formulated in \cite{gilles} for a certain class of polynomial maps, without the restriction to species formation rate functions and conservation laws associated to networks. 
\end{remark}

From Definition \ref{def:deg} and Theorem \ref{injecclose} we obtain the following corollary.
\begin{corollary} \label{inj-deg}
 Let $\mathcal{N}$ be a network. If $\mmN$ is injective then there exist no degenerate positive steady states. 
\end{corollary}

\begin{remark} It follows from Theorem~\ref{injecclose} and Remark~\ref{kinetics-stoich} that if the two spaces $\Gamma$ and $\Lambda_{\kappa}$ are not identical for some rate vector $\kappa$ then the network cannot be injective.
\end{remark}

As noticed in Example~\ref{ex:jac}, the determinant of the Jacobian of the extended species rate formation function of the futile cycle can never vanish. Thus, the futile cycle is injective and does not have the capacity for multiple positive steady states.

The coefficients of the determinant of the Jacobian of fully open networks are characterized by Craciun and Feinberg in \cite{craciun-feinbergI} and this characterization  easily generalizes to open networks. Thus, in order to characterize the coefficients of the determinant of the Jacobian of an arbitrary network $\mmN$, we consider the associated fully open network, $\mmN^o$, and ``match'' the terms of the respective determinants. This is done in \S\ref{closed} after we discuss  some results about open networks in the next section.

\section{Injective open chemical reaction networks}\label{sec6}
Recall that a network is open if its stoichiometric space is $\R^n$. If this  is the case then $\widetilde{f}_{\kappa}=f_{\kappa}$. By Theorem \ref{injecclose}, an open network $\mmN$   is injective if and only if the Jacobian $J_c(f_{\kappa})$ is non-singular, i.e., $\det(J_c(f_{\kappa}))\neq 0$ for all rate vectors $\kappa\in \R^{\mmR}_+$ and all $c\in \R^n_+$. 
Hidden in the proof of Theorem 3.1 in  \cite{craciun-feinbergI} and Theorem \ref{injecclose} above we find a simplification of the characterization of injective open networks: for a network to be injective it suffices to fix any concentration vector $c\in \R^n_+$ and show that $\det(J_c(f_{\kappa}))\neq 0$, for all rate vectors $\kappa\in \R^{\mmR}_+$.

\begin{proposition}\label{simplconc}
An open network  $\mmN$ is  injective if and only if $\det(J_c(f_{\kappa}))\neq 0$ for all rate vectors $\kappa\in \R^{\mmR}_+$ and  a fixed positive $c\in \R^n_+$.
\end{proposition}
\begin{proof} By Theorem \ref{injecclose}, it is enough to prove that 
 $\det(J_c(f_{\kappa}))\neq 0$ for all rate vectors $\kappa\in \R^{\mmR}_+$ and all  $c\in \R^n_+$ if and only if the statement holds for a fixed $c\in \R^n_+$. The forward implication is obvious. 
To see the reverse,  assume that $\det(J_x(f_{\kappa}))= 0$ for some rate vector $\kappa\in \R^{\mmR}_+$ and $x\in \R^n_+$, that is,  there exists a non-zero vector $\gamma \in \R^n$ such that $0 = J_x(f_{\kappa})(\gamma)$.
Define $\eta = (\eta_{y\rightarrow y'})_{y\rightarrow y'} \in \R^{\mmR}_+$ with $\eta_{y\rightarrow y'}=k_{y\rightarrow y'} x^y/c^y>0$, and $\delta=(\delta_1,\dots,\delta_n)$ with $\delta_i = \gamma_ic_i/x_i$. With these definitions, we have $k_{y\rightarrow y'}x^y = \eta_{y\rightarrow y'} c^y$ and $y *_x \gamma =y*_c\delta$. 
Then, using \eqref{jacform}, we have
$$J_c(f_{\eta})(\delta)=  \sum_{y\rightarrow y'\in \mmR} \eta_{y\rightarrow y'} c^y (y *_c \delta) (y'-y)=J_x(f_{\kappa})(\gamma)=0,$$
 which implies that  $\det(J_c(f_{\eta}))=0$ and we have reached  a contradiction. 
\end{proof}

The simplification presented here applies to any open network, independently of whether it contains outflow reactions or not.  We have proved that  injectivity of an open network can be checked using the Jacobian criterion with the concentration vector  fixed to $\mathbf{1}:=(1,\dots,1)$. In this case the determinant of the Jacobian is a polynomial depending only on the rate constants and this reduces the number of variables substantially. Further, the polynomial is homogeneous of total degree $n$.

In \cite[Rk. 2.9]{craciun-feinbergI} a different simplification is performed where the rate constants of all species outflow reactions are fixed to $1$. We state this result as a proposition below and give the proof for completeness. 

\begin{proposition}\label{simplrate} Let $\mmN$ be an open network and fix a subset $\mathcal{O}\subseteq \{1,\dots,n\}$. The following statements are equivalent:
\begin{enumerate}[(i)]
\item   $\mmN$ is  injective.
\item $\det(J_c(f_{\kappa}))\neq 0$ for all $c\in \R^n_+$ and all rate vectors $\kappa$ satisfying $k_{S_i\rightarrow 0}=1$ for all $i\in \mathcal{O}$.
\item For any distinct $a,b\in \R^n_+$, $f_{\kappa}(a)\neq f_{\kappa}(b)$ for all  rate vectors $\kappa$ satisfying $k_{S_i\rightarrow 0}=1$ for all $i\in \mathcal{O}$.
\end{enumerate}
\end{proposition}
\begin{proof} (i) implies (ii) and (iii) by Definition \ref{injective} and Theorem \ref{injecclose}. We now prove that (ii) and (iii) separately imply (i). 
For any $\kappa\in \R^{\mmR}_+$ and $z\in \R^n$ define
  $\widetilde\kappa = (\widetilde k_{y\rightarrow y'})_{y\rightarrow y'}\in \R^{\mmR}_+$ and  $\widetilde z \in \R^n$ by:
 \begin{align*}
 \widetilde z_i &= \begin{cases} z_i k_{S_i\rightarrow 0}, & i\in \mathcal{O} \\  z_i & i\notin \mathcal{O}\end{cases}  &
   \widetilde k_{y\rightarrow y'} & = \frac{ k_{y\rightarrow y'}}{ \prod_{i\in \mathcal{O}} k_{S_i\rightarrow 0}^{y_i}}.
  \end{align*}
  Clearly $\widetilde k_{S_i\rightarrow 0} =1$ for all $i\in \mathcal{O}$, 
  $\widetilde\kappa_{y\rightarrow y'} (\widetilde a)^y = k_{y\rightarrow y'} a^y$, and further $y\ast_{a} \gamma = y \ast_{\widetilde{a}} \widetilde{\gamma}$ for any $\gamma\in \R^n$. It follows that
$f_{\widetilde\kappa}(\widetilde a) = f_{\kappa}(a)$, and, similarly, for any $\gamma\in \R^n$, 
$ J_{\widetilde a}(f_{\widetilde \kappa})(\widetilde\gamma) =  J_{a}(f_{\kappa})(\gamma)$. 
Therefore, if   $\mmN$ is not injective, then (1) there exists a rate vector $\kappa\in \R^{\mmR}_+$ and distinct $a,b\in \R^n_+$  such that $f_{\kappa}(a)=f_{\kappa}(b)$ and the construction of  $\widetilde \kappa\in \R^{\mmR}_+$ and $\widetilde a,\widetilde b\in \R^n_+$ as above implies that (iii) does not hold, and (2) there exist a rate vector $\kappa\in \R^{\mmR}_+$,  $a\in \R^n_+$  and $\gamma \in \R^n$, $\gamma\neq 0$, such that $J_{a}(f_{\kappa})(\gamma)=0$   and thus   a rate vector $\widetilde \kappa\in \R^{\mmR}_+$,  $\widetilde a\in \R^n_+$  and $\widetilde\gamma \in \R^n$ contradicting (ii).  
It follows that  (ii) and (iii) both  imply (i) and the proof is completed. \end{proof}

The simplifications in Propositions \ref{simplconc} and \ref{simplrate} cannot be performed  at the same time because it  would constrain the choice of the free variables $\kappa$ and $c$ too much. 

The next proposition is an extension of   \cite[Eqn. 3.15]{craciun-feinbergI}  where the statement is made for a specific class of open networks. However,  the proof works line by line for the class of all open networks. In fact it does not depend on whether the network is open or not, but all terms in the proposition are zero unless the network is open.  Notice however that our statement differs from the statement  by Craciun and Feinberg in \cite{craciun-feinbergI} in the sign $(-1)^n$, because they establish the Jacobian criterion on the Jacobian of $-f_{\kappa}$.  

Recall that $n$ is the number of species and note that $\det(J_c(f_{\kappa}))$ is a homogeneous polynomial of degree $n$ in the rate constants.
To simplify the notation  we introduce the following: for any set of $m$ reactions, $R=\{y^1\rightarrow y'^{1},\dots,y^m\rightarrow y'^{m}\}$, let \label{yGamma}
\begin{itemize}
\item  $\mmY(R)$ be the $n\times m$ matrix whose $i$-th column is $y^i$.
\item  $\Gamma(R)$ be the $n\times m$ matrix whose $i$-th column is $y^i-y'^i$.  
\item If $m=n$, define $\sigma(R)=(-1)^n \det(\mmY(R))\det(\Gamma(R))$. 
\end{itemize}

\begin{proposition}[\cite{craciun-feinbergI}]\label{coefs} Let $\mmN$ be an open network and let $R=\{y^1\rightarrow y'^{1},\dots,y^n\rightarrow y'^{n}\}$ be a set of $n$ reactions. Viewed as a polynomial in the rate vector $\kappa$,  the coefficient of the monomial $\prod_{i=1}^nk_{y^i\rightarrow y'^i}$ in $\det(J_{c}(f_{\kappa}))$ with $c\in \R^n_+$ is  
$$\alpha(R):= \sigma(R) c^{-\mathbf{1} + \sum_{i=1}^n y^i}.$$ 
In particular, the coefficient of the monomial $\prod_{i=1}^nk_{y^i\rightarrow y'^i}$ in $\det(J_{\bn}(f_{\kappa}))$ is  $\sigma(R)$.
\end{proposition}

\begin{remark} The term 
$\alpha(R)$ is  a monomial in $c$: if $\sum_{i=1}^n y^i_j=0$, then $S_j$ has zero coefficient in $y^i$ for all $i$ and thus $0=\det(y^1,\dots,y^n)=\sigma(R)$.
If a reaction $y^i\rightarrow y'^i$ appears twice in a set $R$ then $\sigma(R)=0$. Therefore the degree of each $k_{y\rightarrow y'}$ in the polynomial  $\det(J_{c}(f_{\kappa}))$ is either zero or one.
\end{remark}

\begin{proposition}\label{openpos}
An open network  $\mmN$ is  injective if and only if the non-zero coefficients $\sigma(R)$   have the same sign for  all sets $R$ of $n$ reactions   and there exists at least one set for which $\sigma(R)\neq 0$.
\end{proposition}
\begin{proof}
The reverse implication follows from Corollary \ref{extendedcor}. For the forward implication, assume that $\mmN$ is injective. Clearly, $\det(J_{\bn}(f_{\kappa}))\neq 0$ and thus there exists at least one set $R$ for which $\sigma(R)\neq 0$. 
Note that  $\det(J_{\bn}(f_{\kappa}))$ is a polynomial in $\kappa$ with total degree $n$ and   degree at most one in each variable $k_{y\rightarrow y'}$. Assume that there exist two coefficients $R_1,R_2$ satisfying $\sigma(R_1)>0$ and  $\sigma(R_2)<0$.  Set $k_{y\rightarrow y'}=1$ if $y\rightarrow y'\notin R_1$ and $k_{y\rightarrow y'}=k$ if $y\rightarrow y'\in R_1$, where  $k$ is a positive parameter. After this transformation, the monomials corresponding to sets of reactions $R\neq R_1$ have degree in $k$ strictly lower than $n$. Then $\det(J_{\bn}(f_{\kappa}))$ is a polynomial of degree $n$ in  $k$, with positive leading coefficient. It follows that if $k$ tends to $+\infty$, then $\det(J_{\bn}(f_{\kappa}))>0$. Symmetrically, using $R_2$ we can find rate constants for which 
 $\det(J_{\bn}(f_{\kappa}))<0$. Since $\det(J_{\bn}(f_{\kappa}))$ is continuous in a connected domain, there exists a rate vector for which $\det(J_{\bn}(f_{\kappa}))=0$, contradicting Proposition \ref{simplconc}.
\end{proof}

 The criterion in Proposition~\ref{openpos} is independent of the rate vector $\kappa$. The requirement that there exists at least one set $R$ for which $\sigma(R)\neq 0$ cannot be removed. Consider for example the  network $\mmN$  with set of reactions $\{A\rightarrow B, A\rightarrow 0\}$. The stoichiometric space has dimension $2$ and thus $\mmN$ is open. However, $\det(J_{c}(f_{\kappa}))= 0$ for all $\kappa,c$, since $c_B$ is not a variable of $f_{\kappa}$. This requirement is not imposed in \cite[Th. 3.3]{craciun-feinbergI} because it  holds  automatically for fully open networks: in fact, the set of reactions $\{S_1\rightarrow 0,\dots,S_n\rightarrow 0\}$  provides the non-zero coefficient $(-1)^n$ in the determinant expansion of the Jacobian. It is mentioned in  \cite[Rk 3.5]{craciun-feinbergI} that the requirement is necessary if the network is not fully open.

\section{Injective closed chemical reaction networks} \label{closed}
We would like to have a characterization of the coefficients of the polynomial $\det(J_{c}(\widetilde{f}_{\kappa}))$ for closed networks similar to that of Proposition \ref{coefs} and a characterization of injectivity similar to that of Proposition \ref{openpos}.  To this end we  consider the fully open network associated to a network and use the results of the previous section. 

Before getting into  technicalities we illustrate the idea with the futile cycle.   The futile cycle has no outflow reactions and therefore the fully open network has an extra reaction $S_i\rightarrow 0$ for all $i$. Let $g_{\kappa^o}$ denote the species formation rate function of the fully open futile cycle with $\kappa^o$ any rate vector such that $k_{S_i\rightarrow 0}=1$. Then, the system of ODEs of the fully open futile cycle taken with mass-action kinetics is:
\begin{align*}
\dot{c_{1}} &=  - k_1 c_{1}c_3+ (k_2+k_3)c_5-c_1   & \dot{c_{2}} &= - k_4 c_{2}c_4 +(k_5+k_6)c_6-c_2  \\
\dot{c_3} &= - k_1c_{1}c_3+ k_2c_5+k_6c_6-c_3  & \dot{c_4} &=-k_4 c_{2}c_4 +k_3c_5+k_6c_6-c_4   \\
\dot{c_5} &=  k_1  c_{1}c_3-(k_2+k_3)c_5-c_5    & \dot{c_6} &=  k_4 c_2c_4-(k_5+ k_6)c_6-c_6. 
\end{align*}
The Jacobian $J_{c}(g_{\kappa^o})$ of $g_{\kappa^o}$ is 
$$
\left(\begin{array}{cccccc}
-k_1  c_3-1 & 0 & -k_1 c_{1} & 0 & k_2+k_3 & 0 \\  0 & -k_4c_4-1 & 0 &  -k_4c_2 & 0 &  k_5+ k_6  \\   - k_1c_3 & 0 &  - k_1c_{1}-1 & 0 & k_2 & k_6 \\  0 &   -k_4c_4 & 0 &   -k_4 c_{2}-1 & k_3 & k_6 \\ k_1  c_3 & 0 & k_1 c_{1} & 0 & -k_2-k_3-1 & 0 \\  0 & k_4c_4 & 0 &  k_4c_2 & 0 &  -k_5- k_6-1 
\end{array}\right).
$$
Propositions \ref{simplrate} and \ref{coefs} link the properties of the coefficients of the determinant of $J_{c}(g_{\kappa^o})$  to the injectivity of the fully open futile cycle. The determinant of $J_{c}(g_{\kappa^o})$ does not change if the fifth row is added to the first, the sixth row to the second, and the fourth, fifth and sixth to the third. Thus,
$$
\det(J_{c}(g_{\kappa^o})) = - \left|\begin{array}{cccccc}
1 & 0 & 0 & 0 & 1 & 0 \\ 0 & 1 & 0 & 0 & 0 & 1 \\  0 & 0 & 1 & 1 & 1 & 1 \\  0 &   -k_4c_4 & 0 &   -k_4 c_{2}-1 & k_3 & k_6 \\ k_1  c_3 & 0 & k_1 c_{1} & 0 & -k_2-k_3-1 & 0 \\  0 & k_4c_4 & 0 &  k_4c_2 & 0 &  -k_5- k_6-1 
\end{array}\right|
$$
where the sign $-$ in front corresponds to changing the sign of the first three rows. This  determinant is  almost equal to the determinant of $J_c(\widetilde{f}_{\kappa})$ as one can see from Example~\ref{futile:jac}. The difference between the two determinants arises from the $-1$ in the diagonal entries of the matrix for the rows $4,5,6$. Therefore, by splitting column $4$ 
using $(0,0,1,-k_4c_2-1,0,k_4c_2)=(0,0,1,-k_4c_2,0,k_4c_2)+(0,0,0,-1,0,0)$,
 and similarly for columns $5,6$, we have
$$
\det(J_{c}(g_{\kappa^o})) = -\det(J_c(\widetilde{f}_{\kappa}))  + \textrm{ monomials of total degree at most 2 in }k_1,\dots,k_6. 
$$
The determinant $\det(J_c(\widetilde{f}_{\kappa}))$ is a homogeneous polynomial in $k_1,\dots,k_6$ of  degree $s=3$. Thus, it agrees with the terms in $-\det(J_{c}(g_{\kappa^o}))$ of  total degree $3$:
\begin{align*} 
 -\det(J_{c}(g_{\kappa^o})) &=(c_1c_2+c_1c_4)k_1k_3k_4 + c_1 (k_1 k_3 k_5+k_1 k_3 k_6 ) + (c_1c_2+c_2c_3) k_1 k_4 k_5 \\ & + c_2 (k_2 k_4 k_5 +k_3 k_4 k_5 ) 
 +  (c_1+ c_3) (k_1 k_6+k_1 k_5) + k_2 k_6  + k_3 k_6 +  k_2 k_5   \\ &  + k_3 k_5 + c_2 k_4 k_5 +(c_2+ c_4)(k_2 k_4 +k_3 k_4)    + (c_1 +c_3)(c_2+ c_4) k_1 k_4  \\ &+ c_1 k_1 k_3
+(c_1+ c_3) k_1 + k_2 + k_3  + (c_2+ c_4) k_4  + k_5+ k_6 +1.
\end{align*}
The fully open network is injective since all coefficients of minus the determinant expansion are positive.

In the example, the row modifications done prior to the computation of the determinant were based on the fact that  the futile cycle has conservation laws. This principle holds generally for any closed network.  Let $\mmN^o$  be the fully open network associated with a closed network $\mmN=(\mmS,\mmC,\mmR)$ with stoichiometric space $\Gamma$. Let $\mathcal{O}(\mmN) := \{i| S_i\rightarrow 0\notin \mmR\}$ be the set of indices for which species outflow reactions do not belong to $\mmR$. 
Then $\mmR^o=\mmR \cup \{S_i\rightarrow 0|\, i\in \on\}$. 
For example, if $\mmN$ is the   futile cycle then $\on=\{1,2,3,4,5,6\}$.
If  $\{\omega^1,\dots,\omega^d\}$ is a reduced basis of  $\Gamma^{\perp}$  then by Lemma \ref{nocons} we have that $i\in \mathcal{O}(\mmN)$ for $i=1,\dots,d$. Thus, the cardinality of $\on$ is at least $d$.

We start by relating the species formation rate functions of $\mmN$ and $\mmN^o$.
Given a rate vector $\kappa=(k_{y\rightarrow y'})_{y\rightarrow y'}\in \R^{\mmR}_+$, define the associated  rate vector $\kappa^o=(k_{y\rightarrow y'}^o)_{y\rightarrow y'}\in \R^{\mmR^o}_+$ by setting $k_{S_i\rightarrow 0}=1$ if $i\in \on$ and $k_{y\rightarrow y'}^o=k_{y\rightarrow y'}$ for $y\rightarrow y'\in \mmR$. 

Let $\delta_i=1$  if $i\in \on$ and $\delta_i=0$ otherwise, and let $E_{m}^n$ be the $n\times n$ matrix with zeroes everywhere but $\delta_i$ in the diagonal entries $(i,i)$ for $i=m,\dots,n$.
Then, the species formation rate function $g_{\kappa^o}(c)$ of $\mmN^o$ with rate vector $\kappa^o$  is  $$g_{\kappa^o}(c)=f_{\kappa}(c) -(\delta_1c_1,\dots,\delta_nc_n).$$

\begin{theorem}\label{prop:det} 
Let $\mmN$ be a closed network,  $\{\omega^1,\dots,\omega^d\}$ a reduced basis of  $\Gamma^{\perp}$ and $\widetilde{f}_{\kappa}(c)$  the associated extended rate function. 
For a rate vector $\kappa\in \R^{\mmR}_+$, let $\kappa^o$, $g_{\kappa^o}(c)$, $E^n_m$ 
 and $\delta_i$ be defined as above.  Then,
 $$\det(J_c(g_{\kappa^o}))=(-1)^d \det(J_{c}(\widetilde{f}_{\kappa}) -  E_{d+1}^n).$$ 
 \end{theorem}
\begin{proof}
Since $g_{\kappa^o}(c)=f_{\kappa}(c) -(\delta_1c_1,\dots,\delta_nc_n)$, we have that $J_c(g_{\kappa^o}) = J_c(f_{\kappa}) -  E_{1}^n$. 
We let $\omega^i=(\lambda_1^i,\dots,\lambda_n^i)$. If $\delta_j=0$, then $S_j\rightarrow 0\in \mmR$ and it follows from Lemma \ref{nocons} that $\lambda^i_j=0$ for all $i$. Thus,   
$ (\omega^i)^tE_{1}^n=(\delta_1\lambda_1^i,\dots,\delta_n\lambda_n^i)= \omega^i$
and from Lemma \ref{omegajac}
$$(\omega^i)^t J_c(g_{\kappa^o})  = (\omega^i)^t J_c(f_{\kappa})  - (\omega^i)^t  E_{1}^n = -\omega^i,\qquad \textrm{for }i=1,\dots,d.$$
  Let $P$ be the $n\times n$ matrix  whose $i$-th row is  $-\omega^i$  for $i=1,\dots,d$   and the $i$-th canonical vector $e_i$ for $i=d+1,\dots,n$.
By the choice of $\omega^i$, $\det(P)=(-1)^d$. Further, $P  J_c(g_{\kappa^o})=J_{c}(\widetilde{f}_{\kappa}) - E_{d+1}^n$ and hence, 
$$(-1)^d \det(J_c(g_{\kappa^o}))= \det(PJ_c(g_{\kappa^o}))= \det(J_{c}(\widetilde{f}_{\kappa}) - E_{d+1}^n).$$\end{proof}

\begin{corollary}\label{coeffs}
The determinant expansion of $J_c(\widetilde{f}_{\kappa})$ as a polynomial in $\kappa$ agrees with the terms in the determinant expansion of $(-1)^d\det(J_c(g_{\kappa^o}))$ of total degree $s$.
\end{corollary}
\begin{proof}
For any non-empty set $I\subseteq \{d+1,\dots,n\}$, let $J_{c}^I(\widetilde{f}_{\kappa})$ be the matrix whose $i$-th column equals that of $J_{c}(\widetilde{f}_{\kappa})$ for $i\notin I$ and is the vector   $-\delta_ie_i$ for $i\in I$. Then, by the column multilinear expansion of the determinant, we have  
$$
(-1)^d\det(J_c(g_{\kappa^o}))= \det(J_{c}(\widetilde{f}_{\kappa})) +\sum_{ \emptyset\neq I\subseteq \{d+1,\dots,n\}} \det(J_{c}^I(\widetilde{f}_{\kappa})).
$$
If $c$ is fixed then  $\det(J_{c}(\widetilde{f}_{\kappa}))$ is  a homogeneous polynomial in $\kappa$ of total degree $s$, while the terms $\det(J_{c}^I(\widetilde{f}_{\kappa}))$ are polynomials in $\kappa$ of total degree strictly lower than $s$.  
\end{proof}

Let $\mmN$ be a network. Let $\mmR_s$ be the set of all sets of $n$ reactions formed by the union of a set $R$ containing $s$ reactions in $\mmR$ and $d=n-s$ outflow reactions $S_i\rightarrow 0$ with $i\in \on$.   Since the cardinality of $\on$ is at least $d$, $\mmR_s\neq \emptyset$. 
By Corollary~\ref{coeffs}  and Proposition~\ref{coefs}   the coefficients of  $\det(J_c(\widetilde{f}_{\kappa}))$  as a polynomial in $\kappa$ and $c$ are, up to a sign, $\alpha(R')$ for $R'\in \mmR_s$.  If $i_j$, $j=1,\dots,d$ are the indices for which outflow reactions $S_{i_j}\rightarrow 0$ belong to $R'$, then $\det(\mmY(R'))$ and $\det(\Gamma(R'))$ are simply  the minors of $\mmY(R')$ and $\Gamma(R')$ with the $i_j$-th rows and columns removed. Equivalently, $\det(\mmY(R'))$ and $\det(\Gamma(R'))$ are the minors obtained by removing    the $i_j$-th rows   from $\mmY(R)$ and $\Gamma(R)$, respectively, for $R=R'\cap\mmR$.  The matrices  $\mmY(R)$ and $\Gamma(R)$ are introduced on page~\pageref{yGamma}.

 Let $\mathcal{O}_{d}(\mmN)$ be the set of subsets of $\on$ of cardinality $d$. 
If $M$ is any $n\times s$ matrix and $I\in \mathcal{O}_{d}(\mmN)$, let $M_{I}$ denote the $s\times s$ submatrix of $M$ obtained by removing the $j$-th row for all $j\in I$. The following corollary is a consequence of Proposition \ref{coefs} and the discussion above.

\begin{corollary}\label{coefs2} Let $\mmN$ be a network and $s$ be the dimension of the stoichiometric space. The terms in the expansion of the determinant $\det(J_{c}(\widetilde{f}_{\kappa}))$ are monomials in $\kappa$ of total degree $s$ and linear in each rate constant. Further, let $R=\{y^1\rightarrow y'^{1},\dots,y^s\rightarrow y'^{s}\}$ be a set of $s$ reactions from $\mmN$. The coefficient of the monomial $\prod_{i=1}^sk_{y^i\rightarrow y'^i}$ in $\det(J_{c}(\widetilde{f}_{\kappa}))$ for $c\in \R^n_+$ is  
$$(-1)^{s} c^{-\mathbf{1} + \sum_{i=1}^s y^i} \sum_{I\in  \mathcal{O}_{d}(\mmN)} \det(\mmY(R)_{I})\det(\Gamma(R)_{I}) \prod_{i\in I} c_{i}.$$ 
Alternatively, the coefficient of $\prod_{i=1}^sk_{y^i\rightarrow y'^i}$ can be written as
$$(-1)^{s}  c^{-\mathbf{1} + \sum_{i=1}^s y^i}  \sum_{R'\in \mmR_s, R =R'\cap \mmR} \sigma(R') \prod_{S_i\rightarrow 0\in R'\setminus R}c_{i}.$$
\end{corollary}
 Observe that the vector $\sum_{i=1}^s y^i$ is simply the row sum of the matrix $\mmY(R)$. The corollary reduces to Proposition~\ref{coefs} if $\mmN$ is open.

 \begin{remark} Only the rate constants of outflow reactions that are \emph{not} in $\mmN$, i.e. those in $\on$, are set to one in the associated  rate vector $\kappa^o$. Otherwise the determinant of $J_{c}(\widetilde{f}_{\kappa})$ would not be a homogenous polynomial of total degree $s$ in $\kappa$.  
\end{remark}

\begin{example}
Consider the futile cycle and \eqref{ex:jac}. The coefficient of $k_1k_3k_4$ in $-\det(J_{c}(\widetilde{f}_{\kappa}))$ is  $c_1c_2+c_1c_4$. It corresponds to the reactions $R=\{r_1,r_3,r_4\}$ and the matrices  $\mmY(R)$ and $\Gamma(R)$ are:
 $$\mmY(R) = \left(\begin{array}{ccc} 1 & 0 & 0 \\ 0 & 0 & 1 \\ 1 & 0 & 0 \\ 0 & 0 & 1 \\ 0 & 1 & 0 \\ 0 & 0 & 0 \end{array}\right)\qquad \qquad
\Gamma(R) = \left(\begin{array}{ccc} 1 & -1 & 0 \\ 0 & 0 & 1 \\ 1 & 0 & 0 \\ 0 & -1 & 1 \\ -1 & 1 & 0 \\ 0 & 0 & 1 \end{array}\right). $$ 
The only sets of indices $I\in \mathcal{O}_{3}(\mmN)$   for which the product $\det(\mmY(R)_{I})\det(\Gamma(R)_{I})$ is non-zero are $I=\{1,2,6\}$ and $I=\{1,4,6\}$. These sets give the coefficient  $c_1c_2+c_1c_4$. Since the last row of $\mmY(R)$ is zero, the index $6$ belongs to all  index sets $I$ for which $\det(\mmY(R)_{I})\det(\Gamma(R)_{I})\neq 0$.
\end{example}

\begin{corollary}\label{det-criterion}
Let $\mathcal{N}=(\mmS,\mmC,\mmR)$ be a network. The following are equivalent:
\begin{enumerate}[(i)]
\item $\mmN$ is injective.
\item The non-zero coefficients $\sigma(R')$ have the same sign for all sets $R'\in \mmR_s$,   and $\sigma(R')\neq 0$ for at least one  set $R'$.
\item The non-zero products $\det(\mmY(R)_{I})\det(\Gamma(R)_{I})$ have the same sign for all sets $R$ of $s$ reactions in $\mmR$ and  $I\in \mathcal{O}_{d}(\mmN)$, and further $\det(\mmY(R)_{I})\det(\Gamma(R)_{I})\neq  0$ for at least one set $R$ and some $I\in \mathcal{O}_{d}(\mmN)$.
\end{enumerate}
\end{corollary}
\begin{proof} 
The equivalence between (ii) and (iii) is a consequence of Corollary \ref{coefs2}. If (ii) holds then (i) is a consequence of  Corollary \ref{coefs2} and Corollary \ref{extendedcor}. To show that (i) implies (ii), we use the same argument as in the proof of Proposition \ref{openpos}. Using Corollary \ref{coefs2} and Corollary \ref{extendedcor} it suffices to show that for any set $R'\in \mmR_s$ we can find a concentration vector $c$ and a  rate vector $\kappa$  for which the sign of $\det(J_{c}(\widetilde{f}_{\kappa}))$ agrees with the sign of $(-1)^s \sigma(R')$. 
Let $k$ be a positive parameter and let $m$ be an integer. Let $R=R'\cap \mmR$. Define $k_{y\rightarrow y'}=k$ if $y\rightarrow y'\in R$, and $k_{y\rightarrow y'}=1/k^m$ if $y\rightarrow y'\notin R$. Define $c_l=k$ if $S_l\rightarrow 0\in R$ and $c_l=1$ otherwise.  Then the monomial corresponding to the set $R'\in \mmR_s$  is the only monomial that  tends to $\pm \infty$ when $k$ tends to infinity for $m$ large enough.
\end{proof}

\begin{example}
We consider the extension of the futile cycle to incorporate two modification sites instead of one.  
The network consists of the  reactions 

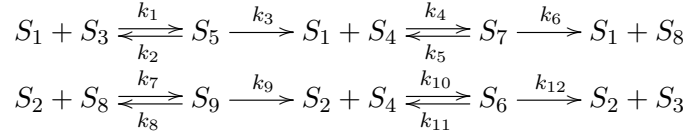
\vspace{0.1cm}
 \centerline{\xymatrix@R=10pt{
S_{1} + S_3 \ar@<0.3ex>[r]^(0.6){k_1}  & S_5  \ar@<0.3ex>[l]^(0.4){k_2}  \ar[r]^(0.4){k_3} & S_{1} + S_4 \ar@<0.3ex>[r]^(0.6){k_4}  & S_7 \ar@<0.3ex>[l]^(0.4){k_5}  \ar[r]^(0.4){k_6} & S_1+S_8   \\
S_{2} + S_8 \ar@<0.3ex>[r]^(0.6){k_{7}}  & S_9  \ar@<0.3ex>[l]^(0.4){k_{8}}  \ar[r]^(0.4){k_{9}}  & S_{2} + S_4  \ar@<0.3ex>[r]^(0.6){k_{10}}  & S_6 \ar@<0.3ex>[l]^(0.4){k_{11}}  \ar[r]^(0.4){k_{12}} & S_2+S_3
}}

\noindent with rate constants indicated next to each reaction.
The enzyme $S_1$ catalyzes the modification of $S_3$ to $S_4$ and subsequently to $S_8$ via the formation of the intermediates $S_5,S_7$. Similarly, the enzyme $S_2$ catalyzes the demodification of $S_8$ to $S_4$ and then to $S_3$ via the intermediates $S_9$ and $S_6$. 
A reduced basis of $\Gamma^{\perp}$ of this network is given by the vectors $S_1+S_5+S_7$,  $S_2+S_6+S_9$, and$S_3+S_4+S_5+S_6+S_7+S_8+S_9$. 
The extended rate function is obtained by substituting the components of $f_{\kappa}$ with indices $1,2$, and $3$ by $c_1+c_5+c_7$, $c_2+c_6+c_9$, and $c_3+c_4+c_5+c_6+c_7+c_8+c_9$, respectively.

The determinant of the extended rate function has the monomials $$k_1 k_3 k_4 k_7 k_9  k_{12} c_1 c_2 c_3\qquad \textrm{and}\qquad - k_2 k_4 k_6 k_7 k_{10}k_{12} c_1 c_2 c_4.$$  
The two terms have different signs.  Therefore, this network is not injective. Note that the degree of the monomials in the rate constant is $s=6$.

It is well known that this network can exhibit multistationarity for some choices of rate constants and total amounts (\cite{Markevich-mapk}). More generally, in \cite{Feliu:royal} Feliu and Wiuf analyzed the occurrence of multistationarity in different smalls motifs accounting for enzyme sharing in protein modification, including the futile cycle and the two-site modification cycle. In their examples, all motifs that admit exactly one positive steady state for any total amounts are in fact injective. The motifs that can admit multiple positive steady states are obviously not injective.
\end{example}

\begin{example}\label{ex:noninj}
Being injective is not a necessary condition for the existence of at most one positive steady state within each stoichiometric class.
Consider the network with reactions
 
 
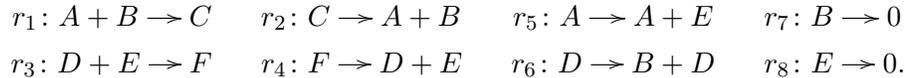
\centerline{
 \xymatrix@C=12pt@R=2pt{
r_1\colon A+B \ar[r]  & C & r_2\colon C \ar[r] & A+B &  r_5\colon A \ar[r] & A+E & r_7\colon B\ar[r] & 0 \\  r_3\colon D+E \ar[r]  & F  & r_4\colon F \ar[r] & D+E   & r_6\colon D\ar[r] & B+D & r_8\colon E \ar[r] & 0. &}}

\vspace{0.1cm}
\noindent
For $R_1=\{r_2,r_4,r_7,r_8\}$, we have $\sigma(R_1)=1$, while for  $R_2=\{r_1,r_3,r_5,r_6\}$, we have $\sigma(R_2)=-1$. It follows that the network is not injective. However, by solving the steady-state equations together with the equations for the conservation laws, it is easily seen that there is exactly  one positive steady state in each stoichiometric class.
\end{example}

\section{Networks with all steady states degenerate}\label{sec8}
If $\det(J_{c}(\widetilde{f}_{\kappa}))$ is not identically zero then  there exist $s$ linearly independent reaction vectors   $y^1- y'^{1},\dots,y^s- y'^{s}$  such that  $y^1,\dots,y^s$ are also linearly independent.
Therefore, if the dimension of the vector space $$\mmY:=\langle y\in \mmC| y\textrm{ is the reactant complex of some reaction }y\rightarrow y' \rangle \subseteq \R^n$$ is strictly smaller than $s$ then $\det(J_{c}(\widetilde{f}_{\kappa}))=0$. For example, if the reactant complexes of a network involve at most $s-1$ species, then all steady states of the network are degenerate.  The network with reactions $\{A\rightarrow B,A\rightarrow 0\}$ satisfies $\dim \mmY=1<2$, consistent with our computation that $\det(J_{c}(\widetilde{f}_{\kappa}))=\det(J_c(f_{\kappa}))=0$.

Let $I\in \mathcal{O}_{d}(\mmN)$ and let $\mmS_{I} =\{S_i\in \mmS|\ i\notin I\}$ be the set of species with indices not in $I$. Note that $\mmS_I$ has cardinality $s$. We consider the \emph{projection} of the network $\mmN$ to the set of species $\mmS_{I}$, $\mmN_I$,   induced by the projection $\pi_I: \R^n\rightarrow \R^s$ on the coordinates not in $I$.
 For example, the projection of the futile cycle with $I=\{1,2,6\}$ is the network  with species $\mmS_I=\{S_3,S_4,S_5\}$ and reactions 
 
\centerline{\xymatrix{
  S_3 \ar@<0.3ex>[r]  & S_5  \ar@<0.3ex>[l]  \ar[r] &  S_4  &
 S_4 \ar@<0.3ex>[r]  & 0  \ar@<0.3ex>[l]  \ar[r]  &  S_3.
}}
\noindent If $0\rightarrow 0$ occurs then the reaction is discarded and like-wise redundant reactions are removed. The matrix $\Gamma(R)_I$ varies  over all sets of $s$ reactions in $\mmN_I$ as $R$ varies. Similarly, $\mmY(R)_I$ varies over all sets of $s$ reactant complexes in $\mmN_I$ as $R$ varies.  Therefore, the requirement that $\det(\mmY(R)_{I})\det(\Gamma(R)_{I})\neq  0$ for some $R$   is equivalent to  the existence of $s$ independent reactions in $\mmN_I$ such that the corresponding reactant vectors also are independent.  Since $\mmN_I$ has $s$ species, a necessary condition is that $\mmN_I$ is open, that is, the stoichiometric space has dimension $s$.  In the example above, the set of reactions $\{S_3\rightarrow S_5,S_5\rightarrow S_4,S_4\rightarrow 0\}$ are independent and so are the reactant complexes. This implies (as also shown in Example~\ref{futile:jac}) that the steady states of the futile cycle are not degenerate.

The following corollary is a consequence of Corollary \ref{coefs2}, Corollary \ref{jaccomp} and the preceding discussion.

\begin{corollary}\label{deg}
Let $\mathcal{N}=(\mmS,\mmC,\mmR)$ be a network. The following statements are equivalent:
\begin{enumerate}[(i)]
\item $\ker(J_{c}(f_{\kappa}))\cap \Gamma \neq \{0\}$ for all $\kappa\in \R^{\mmR}_+$ and $c\in \R^n_+$.
\item $\det(\mmY(R)_{I})\det(\Gamma(R)_{I})= 0$ for all sets  $R$ of $s$ reactions  in $\mmR$ and $I\in \mathcal{O}_{d}(\mmN)$.
\item For all $I\in \mathcal{O}_{d}(\mmN)$ and for any set of $s$ reactions $y^1\rightarrow y'^{1},\dots,y^s\rightarrow y'^{s}$ in the projected network $\mmN_I$, if the vectors  $y^1- y'^{1},\dots,y^s- y'^{s}$ are linearly independent then the complexes $y^1,\dots,y^s$ are linearly dependent.
\end{enumerate}
If any of these hold then $\mmN$ is not injective and  any steady state is degenerate.
\end{corollary}

If $\mmN$ is open, then $d=0$ and the only projection to consider is the identity. Therefore, condition (iii) reduces to the condition of the following corollary.

\begin{corollary}\label{deg2}
Let $\mathcal{N}=(\mmS,\mmC,\mmR)$ be an open network. Then $\ker(J_{c}(f_{\kappa}))\cap \Gamma \neq \{0\}$ for all $\kappa\in \R^{\mmR}_+$ and $c\in \R^n_+$ if and only if for any set of $n$ reactions $y^1\rightarrow y'^{1},\dots,y^n\rightarrow y'^{n}$ such that the vectors  $y^1- y'^{1},\dots,y^n- y'^{n}$ are linearly independent, the complexes $y^1,\dots,y^n$ are linearly dependent.
\end{corollary}

If $\mmN$ is fully open then the set of species outflow reactions provides a set of independent reaction vectors and independent reactant complexes. Therefore, fully open networks cannot fulfill that $\ker(J_{c}(f_{\kappa}))\cap \Gamma \neq \{0\}$ for all $\kappa\in \R^{\mmR}_+$ and $c\in \R^n_+$. However,  open networks that are not fully open might fulfill the condition.  For example, consider  the network with reactions 
$$r_1\colon S_1 \rightarrow S_3\qquad  r_2\colon S_3 \rightarrow S_1\qquad r_3\colon   S_1+S_2 \rightarrow  S_3 \qquad r_4\colon 2S_1+2S_2 \rightarrow  S_3.
$$
The dimension of $\mmY$ and $\Gamma$ agree ($s=3$) and the network is open. The network has no outflow reactions. Reactions $r_1,r_2$ are linearly dependent and thus any set of $3$ independent reactions must contain $r_3$ and $r_4$. In that case, however, the reactant complexes are linearly dependent. It follows that Corollary \ref{deg2}(ii) is fulfilled and hence all steady states of the network are degenerate.

\begin{remark} If $\mathcal{Y}\subsetneq\Gamma$, in which case $\dim(\mmY)<s$ and Corollary~\ref{deg}(ii) is fulfilled, then all stoichiometric classes have  either none or infinitely many positive steady states: let $a\in\R^n_+$ be a steady state. Then there is a vector $\gamma\in\Gamma$ such that $\gamma\cdot y=0$ for all $y\in\mathcal{Y}$. Consequently,  for all $t\in\R$ and  $\mu\in\Gamma^{\perp}\subseteq\mathcal{Y}^{\perp}$  we have $e^{y\cdot(t\gamma+\mu)}=1$ and hence $c_{t,\mu}=a e^{t\gamma+\mu}$ is a steady state. Further,  $c_{t,\mu}\in G_t=\{ c\in\R^n_+| \log(c)-\log(ae^{t\gamma})\in\Gamma^{\perp}\}$ which intersects each stoichiometric class in exactly one point for each $t$ \cite{Feinbergss}. It is easy to prove that $G_{t}\cap G_{t'}=\emptyset$ if $t\neq t'$ and hence there are infinitely many positive steady states in each stoichiometric class.
\end{remark}

 \begin{remark} In \cite{craciun-feinberg-semiopen}, a related determinant criterion is given to decide whether or not a closed network can admit degenerate steady states.  Fix a determinant function $\det_{\Gamma}$ on the  stoichiometric space $\Gamma$. For each choice of scalar product $*_c$ in equation \eqref{jacform} (that is, for each choice of  $c\in\R_+^n$)  non-singularity of $J_c(f_{\kappa})$ restricted to $\Gamma$  is related to a polynomial expansion in $\kappa$ of  $\det_{\Gamma}(J_c(f_{\kappa}))$. Each term in the determinant expansion takes the form in Corollary~\ref{det-criterion} with `$\det$' replaced by $\det_{\Gamma}$ and $y\in\mathcal{Y}(R)_I$ replaced by $\pi_c(y)$, where $\pi_c$ is the projection onto $\Gamma$ as defined by the scalar product $*_c$. The criterion requires the coefficients of the terms in the determinant expansion to be of  the same sign or zero for each choice of $c\in\R^n_+$ \cite[Prop. 10.3]{craciun-feinberg-semiopen}. It is not obvious how to check whether this criterion is fulfilled using computational algebra software, in contrast to the criterion in Proposition~\ref{det-criterion}.
\end{remark}

\begin{remark} For fully open networks, an algorithm is provided Joshi and Shiu in \cite{joshi-shiu-I} to simplify the search for coefficients of the determinant of the Jacobian that have the ``wrong'' sign. The algorithm can be applied in the present setting, that is, to closed networks, to restrict  the sets of $s$ reactions to consider.
\end{remark}

\begin{remark}
The projected networks are  \emph{embedded networks} as defined by  Joshi and Shiu in \cite{joshi-shiu-II}. We have shown that  for an injective network all  embedded networks obtained by selecting sets of $s$ species as above are either injective or have all steady states degenerate. Consequently, if a network is injective then the embedded networks with $s$ species do not have the capacity to admit multiple non-degenerate steady states.
\end{remark}

\section{Open and closed networks and injectivity}\label{sec9}
In \cite{craciun-feinberg}, Craciun and Feinberg  preclude multistationarity in  closed networks provided that (1) the fully open network is injective and (2) the closed network does not have degenerate steady states. In a later paper \cite[Th. 8.2]{craciun-feinberg-semiopen}, the authors provide a sufficient condition (namely that the \emph{entrapped-species projection is a normal reaction network}) for (2) to hold provided (1) holds as well. All weakly reversible networks fulfill this condition. 

 Using the results of the previous sections we now relate injectivity of a network  $\mmN$ and injectivity of the associated fully open network $\mmN^o$.

\begin{theorem}\label{th:closeinj}
Let $\mathcal{N}=(\mmS,\mmC,\mmR)$ be a closed network.
If $\mathcal{N}^o$ is injective then the following statements are equivalent:
\begin{enumerate}[(i)]
\item $\mmN$ is injective.
\item As a polynomial in $\kappa$, $\det(J_c(g_{\kappa^o}))$  has at least one monomial of total degree $s$.
\item $\ker(J_{c}(f_{\kappa}))\cap \Gamma = \{0\}$ for some fixed $\kappa\in \R^{\mmR}_+$ and $c\in \R^n_+$.
\item There exists $I\in \mathcal{O}_{d}(\mmN)$ and a set of $s$ reactions $y^1\rightarrow y'^{1},\dots,y^s\rightarrow y'^{s}$ in the projected network $\mmN_I$ such that the two sets of vectors  $y^1- y'^{1},\dots,y^s- y'^{s}$ and $y^1,\dots,y^s$ are both linearly independent.
\end{enumerate}
\end{theorem}
\begin{proof} Let $\{\omega^1,\dots,\omega^d\}$ be a reduced basis of  $\Gamma^{\perp}$ and $\widetilde{f}_{\kappa}$ the associated extended rate function for $\kappa\in \R^{\mmR}_+$. Since $\mmN^o$ is injective, then $J_c(g_{\kappa^o})$ is non-singular for all $\kappa\in \R^{\mmR}_+$ and for any $c\in \R^n_+$ (Proposition \ref{simplrate}). Equivalently,  $\det(J_c(g_{\kappa^o}))$ is a polynomial in $\kappa$, with all non-zero coefficients having the same sign and thus all non-zero coefficients of $\det(J_{c}(\widetilde{f}_{\kappa}))$ have the same sign. It follows that either (1) $\det(J_{c}(\widetilde{f}_{\kappa}))=0$ as a polynomial in $\kappa$ or (2) $\det(J_{c}(\widetilde{f}_{\kappa}))\neq 0$ for all values of $\kappa\in \R^{\mmR}_+$ and $c\in \R^n_+$.
The equivalence of the four statements follows from this observation,   Theorem \ref{injecclose}, Corollary \ref{extendedcor} and Corollary \ref{deg}.
\end{proof}

\begin{corollary}
Let $\mathcal{N}=(\mmS,\mmC,\mmR)$ be a closed network, $\mmN^o=(\mmS,\mmC^o,\mmR^o)$ the associated fully open network, and $g_{\kappa}$ the species formation rate function of $\mmN^o$.  Assume  that  $\mmN^o$ is injective. Then $\mmN$ is injective  if and only if  as a polynomial in  $\kappa=(k_{y\rightarrow y'})_{y\rightarrow y'\in \mmR^o}$, $\det(J_{\bn}(g_{\kappa}))$  
has at least one monomial in $k_{y^1\rightarrow y'^{1}},\dots,k_{y^n\rightarrow y'^{n}}$ 
with $s$ reactions in $\mmR$. If this is the case, then all steady states of $\mmN$  are non-degenerate.
\end{corollary}

\begin{example} Consider the open network $\mmN$  with reactions $r_1\colon A+B\rightarrow D$, $r_2\colon A\rightarrow C$, $r_3\colon B\rightarrow D$. The dimension of the stoichiometric space is $3$. Let $k_i$ denote the rate constant of reaction $r_i$. The associated fully open network is injective,
since 
$$\det(J_c(g_{\kappa^o})) = 1 + k_2 + k_3 + k_1 c_A  +k_1c_B+ k_2 k_3 + k_1 k_2 c_A + k_1 k_3 c_B.$$
However, the determinant    has no monomial involving $k_1,k_2,k_3$. It follows that $\mmN$ is not injective and  all steady states are degenerate. Alternatively, note that $\mmY=\langle A,B\rangle$ has dimension $2<s=3$. 
\end{example}

In view of these results, there might exist injective networks such  that their open network counterparts are not injective.  This can only occur if some monomials with less than $s$ rate constants from the true reactions have the wrong sign.  
We provide two examples.

\begin{example}
Consider the network $\mmN$ given by the reactions:
$$r_1\colon A+C\rightarrow 2A+B,\quad r_2\colon B+C\rightarrow A+C,\quad r_3\colon A+B+C\rightarrow 2B.$$
The stoichiometric space of $\mmN$ has maximal dimension $3$ and thus it is open. Let $k_i$ denote the rate constant of reaction $r_i$, and $k_A,k_B,k_C$ the rate constants of the outflow reactions. Then, for $\kappa=(k_1,k_2,k_3,k_A,k_B,k_C)$, 
\begin{multline*}
\det(J_{\mathbf{1}}(g_{\kappa}))= -2k_1k_2k_3 - k_1k_2k_A- k_1k_2k_B-k_2k_3k_B - k_1k_Ak_B -k_3k_Ak_B
+ 2 k_1k_2k_C \\- 2k_1k_3k_C  - k_2k_Ak_C + k_3k_Ak_C + k_1k_Bk_C -k_3k_Bk_C -k_Ak_Bk_C.
\end{multline*}
Since the sign of the coefficients of $\det(J_{\mathbf{1}}(g_{\kappa}))$ as a polynomial in $\kappa$ differ,   the fully open network, $\mmN^o$, is not injective. However, there is  one monomial only in $k_1,k_2,k_3$, which implies that $\mmN$ is injective and has no degenerate steady states.  This network, however, does not have positive steady states either.
\end{example}

\begin{example}
Consider the  network $\mmN$ given by the reactions:
$$r_1\colon A+B+C\rightarrow 2A+B+2C,\quad r_2\colon A+C\rightarrow A+B+C,\quad r_3\colon C\rightarrow A+B+ 2C.$$
The stoichiometric space of this network has dimension $2$ and thus it is closed. Let $k_i$ be the rate constant of reaction $r_i$ and fix the rate constants of the outflow reactions to $k_A=k_B=k_C=1$. Then, for $\kappa^o=(k_1,k_2,k_3,1,1,1)$, 
\begin{multline*}
\det(J_{c}(g_{\kappa^o}))= k_1 k_3 c_Ac_C +k_1 k_2 c_A^2c_C +  k_1 k_2 c_A c_C^2 + k_3 + k_1 c_Ac_B  + k_1 c_Bc_C -1.
 \end{multline*}
 We see that $\mmN^o$ is not injective but since  all monomials in $k_1,k_2,k_3$ of total degree 2 have  the same sign, it follows from corollaries \ref{coeffs}, \ref{coefs2} and \ref{det-criterion}  that $\mmN$ is injective and there are no degenerate steady states.
 \end{example}

\begin{remark} The open network given in \cite[$\S$6]{craciun-feinberg} to illustrate that degenerate steady states can occur is not injective. Therefore the results of this work do not apply.
\end{remark}

\begin{remark} This remark is for the readers familiar with  \cite{craciun-feinberg-semiopen}  and the terminology introduced in that paper. Let $\mmN$ be a network such that $\mmN^o$ is injective. We have provided a sufficient and necessary condition for $\mmN$ to be injective as well, namely, that there exists $I\in \mathcal{O}_{d}(\mmN)$, and a set of $s$ reactions $y^1\rightarrow y'^{1},\dots,y^s\rightarrow y'^{s}$ in the projected network $\mmN_I$ such that the set of vectors  $y^1- y'^{1},\dots,y^s- y'^{s}$ and $y^1,\dots,y^s$ are both linearly independent. Since it is a sufficient and necessary condition, any network fulfilling the condition in \cite[Th. 8.2]{craciun-feinberg-semiopen} for $\mmN$ to be injective, that is,   the  entrapped-species projection is a normal reaction network, also fulfills our condition. In particular,  weakly reversible networks are normal and hence they also fulfill our condition.
\end{remark}

\section{Weakly Sign Determined (WSD) networks}\label{sec10}
A square matrix is said to be a $P$-matrix if all principle minors are positive. If the principle minors are non-negative the matrix is said to be a $P_0$-matrix.  In \cite{Banaji-donnell}, a criterion (namely that the \emph{stoichiometric matrix  is WSD}) is introduced that relates to injectivity of a network $\mmN$.  They restrict the class of networks to \emph{non-autocatalytic networks} (NAC), which implies that the same species cannot both be a reactant and a product in the same reaction. We impose the same constraint in this section.

In our notation the criterion states  that $(*)$ {$\det(\mathcal{Y}(R)_I)\det(\Gamma(R)_I)\geq 0$ for any set $I$ of $n-k$ distinct indices in $\{1,\ldots,n\}$ and any set  $R=\{y^1\rightarrow y'^{1},\dots,y^k\rightarrow y'^{k}\}$ of $k$ reactions  from $\mmN$  if and only if $-J_c(f_{\kappa})$ (minus the Jacobian) is a $P_0$-matrix \cite[Th.~4.1, Th.~4.3]{Banaji-donnell}.  If $\mmN$ is fully open then $(*)$ (with inflow and outflow reactions \emph{excluded} from the sets $R$) is equivalent to $-J_c(f_{\kappa})$ being a $P$-matrix \cite[Cor.~4.2, Th.~4.4]{Banaji-donnell}. Using the results of  \cite{Gale:1965p474}, it follows that  $\mmN$ is injective. 

If $\mmN$ is closed and injective then $J_c(f_{\kappa})$ is not a $P$-matrix as the rank is less than $n$. Our criterion for injectivity  states that if there exists $\sigma\in\{0,1\}$ such that $(**)$ $(-1)^{\sigma}\det(\mathcal{Y}(R)_I)\det(\Gamma(R)_I)\geq 0$   for any set $I$ of $n-s$ distinct indices in $\{1,\ldots,n\}$ and any set  $R=\{y^1\rightarrow y'^{1},\dots,y^s\rightarrow y'^{s}\}$ of $s$ reactions  from $\mmN$, and at least one of them is non-zero, then $\det(J_c(\widetilde{f}_{\kappa}))$ is non-zero and $\mmN$ is injective. 
Assume that $(**)$ holds. If all species in $\mmN$ appear in some reactant complex then the NAC assumption guarantees that the diagonal entries of $J_c(\widetilde{f}_{\kappa})$ in rows $d+1,\dots,n$ are non-zero and negative. Since the diagonal entries in rows $1,\dots,d$ are all equal to one, it follows that the product of the diagonal entries is a term in the expansion of $\det(J_c(\widetilde{f}_{\kappa}))$ and has sign $(-1)^s$. Consider the matrix  $J^{*}_c(\widetilde{f}_{\kappa})$ obtained from $J_c(\widetilde{f}_{\kappa})$ by multiplying the lower $s$ rows by minus one. The matrix has full rank and is a $P$-matrix:
if $\mmN$ fulfills $(**)$ then all non-zero terms in the determinant expansion of $J^{*}_c(\widetilde{f}_{\kappa})$ have sign $(-1)^{s}$. Each principle minor can be obtained as a sum of terms in the determinant expansion of $J^{*}_c(\widetilde{f}_{\kappa})$  divided by the (positive) diagonal entries of the rows not taken in the minor. Since this sum contains at least the non-zero diagonal product term, all principle minors are positive.  
Consequently, $J^{*}_c(\widetilde{f}_{\kappa})$ is a $P$-matrix.  Note that if $\mmN$ has full rank (for example if it is fully open) then $J^{*}_c(\widetilde{f}_{\kappa})=-J_c(\widetilde{f}_{\kappa})$.

If $S_i$ is a species that is only in product complexes of $\mmN$ then the lower $s$ rows of $J_c(\widetilde{f}_{\kappa})$ are zero in the $i$-th position. If $S_i$ is not involved in any conservation law, then the $i$-th column of $J_c(\widetilde{f}_{\kappa})$ is zero, $\det(J_c(\widetilde{f}_{\kappa}))= 0$, and $(**)$ does not hold. Generally, let 
$S_{i_1},\dots,S_{i_m}$ be the species of $\mmN$ that are only in reactant complexes such that  the lower $s$ rows of $J_c(\widetilde{f}_{\kappa})$ are zero in entries $i_1,\dots,i_m$.
If $(**)$ holds and $\det(J_c(\widetilde{f}_{\kappa}))\neq 0$ then the columns $i_1,\dots,i_m$ of $J_c(\widetilde{f}_{\kappa})$ are linearly independent.
   Since the lower $s\times m$ submatrix is identically zero, it follows that the upper $d\times m$ matrix has rank $m$ (and in particular $m\leq d$). As a consequence, we can reorder the species in $\mmS$ such that $S_{i_1},\dots,S_{i_m}$ are the first $m$ species and we are guaranteed that there exists a reduced basis of $\Gamma^{\perp}$ with that order. Because $m\leq d$, with this order the diagonal entries of $J_c(\widetilde{f}_{\kappa})$ are non-zero and we can proceed as above.

To sum up, if $(**)$ holds, then there exists an order of the species of $\mmN$ such that the matrix obtained from $J_c(\widetilde{f}_{\kappa})$  by changing the sign of the last $s$ rows is a $P$-matrix.  Using the results of \cite{Gale:1965p474} we  conclude  that $\widetilde{f}_{\kappa}$ is an injective function and hence that  $\mmN$ is injective.

\section{Concluding remarks}
\begin{figure}
\begin{center}
\begin{tikzpicture}[scale=1.3]
\node[draw=yellow!70!green,diamond,aspect=2.5, rounded corners,line width=1pt,nearly transparent,fill=yellow!70!green] (t0) at (2,4) {$\mmN$ network};
\node[draw=orange,diamond,aspect=2.5, rounded corners,line width=1.5pt] (t0) at (2,4) {$\mmN$ network};
\node[draw=yellow!70!green,diamond,aspect=2.5, rounded corners,line width=1pt,nearly transparent,fill=yellow!70!green] (t3) at (6,0.5) {$\mmN$ injective};
\node[draw=orange,diamond,aspect=2.5, rounded corners,line width=1.5pt] (t3) at (6,0.5) {$\mmN$ injective};
\node[draw=yellow!70!green,diamond,aspect=2.5, rounded corners,line width=1pt,nearly transparent,fill=yellow!70!green] (t1)  at (-2,1.2) {$\mmN^o$ injective};
\node[draw=yellow!70!green,diamond,aspect=2.5, rounded corners,line width=1pt] (t1)  at (-2,1.2) {$\mmN^o$ injective};

\draw[->,gray,dashed,line width=2pt] (t0) .. controls +(-3.5,0) and +(0,1) .. (t1);
\node[fill=white,text width=3cm,text badly centered,draw=magenta,rectangle, rounded corners,line width=1pt] (E1) at (-1.4,3.2)  {Pass Jacobian criterion for fully open networks?};
\node[rectangle,draw=blue,rounded corners,fill=white] (c1) at (-2,2.3) {YES};

\draw[->,gray,line width=2pt] (t0) .. controls +(3.5,0) and +(0,1) .. (t3);
\node[fill=white,text width=2.8cm,text badly centered,draw=magenta,rectangle, rounded corners,line width=1pt] (E1) at (5.2,3.2)  {Pass Jacobian criterion for any network?};
\node[rectangle,draw=blue,rounded corners,fill=white] (c1) at (5.8,2.1) {YES};

\node[text width=2.75cm,draw=yellow!70!green,rectangle,text badly centered, rounded corners,line width=1pt,nearly transparent,fill=yellow!70!green] (t2)  at (-1,-2) {$\mmN$ cannot have multiple  non-degenerate  steady states};
\node[text width=2.75cm,draw=yellow!70!green,rectangle,text badly centered, rounded corners,line width=1pt] (t2)  at (-1,-2) {$\mmN$ cannot have multiple  non-degenerate  steady states};

\node[text width=3cm,draw=yellow!70!green,rectangle,text badly centered, rounded corners,line width=1pt,nearly transparent,fill=yellow!70!green] (t4)  at (3,-0.5) {$\mmN$ does not have degenerate steady states};
\node[text width=3cm,draw=yellow!70!green,rectangle,text badly centered, rounded corners,line width=1pt] (t4)  at (3,-0.5) {$\mmN$ does not have degenerate steady states};

\node[text width=2.7cm,text badly centered,draw=yellow!70!green,rectangle, rounded corners,line width=1pt,inner sep=5pt,nearly transparent,fill=yellow!70!green] (t5)  at (6,-2) {$\mmN$ cannot have multiple nor degenerate steady states};
\node[text width=2.7cm,text badly centered,draw=orange,rectangle, rounded corners,line width=1.5pt,inner sep=5pt] (t5)  at (6,-2) {$\mmN$ cannot have multiple nor degenerate steady states};

\draw[->,gray,line width=2pt] (t1) .. controls +(1,0.8) and +(-1,1.8) .. (t3);
\node[fill=white,text width=3.8cm,text badly centered,draw=magenta,rectangle, rounded corners,line width=1pt] (E2) at (1.8,1.8)  {Determinant of the Jacobian has one monomial of total degree s?};
\node[rectangle,draw=blue,rounded corners,fill=white] (c1) at (4.7,1.6) {YES};

\draw[->,gray,line width=2pt,dashed] (t1) .. controls +(0,-1.5) and +(-2,0) .. (t4);
\node[fill=white,text width=1.5cm,text badly centered,draw=magenta,rectangle, rounded corners,line width=1pt] (E1) at (-0.8,-0.2)  {Normal projection?};
\node[rectangle,draw=blue,rounded corners,fill=white] (c1) at (0.7,-0.4) {YES};

\draw[->,gray,line width=2pt] (E2) .. controls +(0,-1) and +(0,1) .. (t4);
\node[rectangle,draw=blue,rounded corners,fill=white] (c1) at (2.5,0.6) {YES};

\draw[->,gray,line width=2pt,dashed] (t1) .. controls +(-1,-1.5) and +(0,1) .. (t2);
\draw[->,gray,line width=2pt,dashed] (t2) -- (t5);
\draw[gray,line width=2pt,dashed] (t4) .. controls +(-2,-1.5) and +(0,0.1) .. (3,-2);
\draw[fill=gray,draw=gray] (3,-2) circle (2pt); 

\draw[->,gray,line width=2pt] (t3) -- (t5);
\draw[rounded corners, gray] (-3,-3.5) rectangle (3.3,-2.9);
\draw[-,gray,dashed,line width=2pt] (-2.1,-3.2) -- (-2.8,-3.2);
\node[anchor=west]  (cf) at (-2,-3.2) {Craciun, Feinberg};
\draw[-,gray,line width=2pt] (1,-3.2) -- (1.7,-3.2);
\node[anchor=west]  (cf) at (1.8,-3.2) {This work};

\end{tikzpicture}
\end{center}
\caption{This work in relation to previous work of Craciun and Feinberg.}\label{summary}
\end{figure}
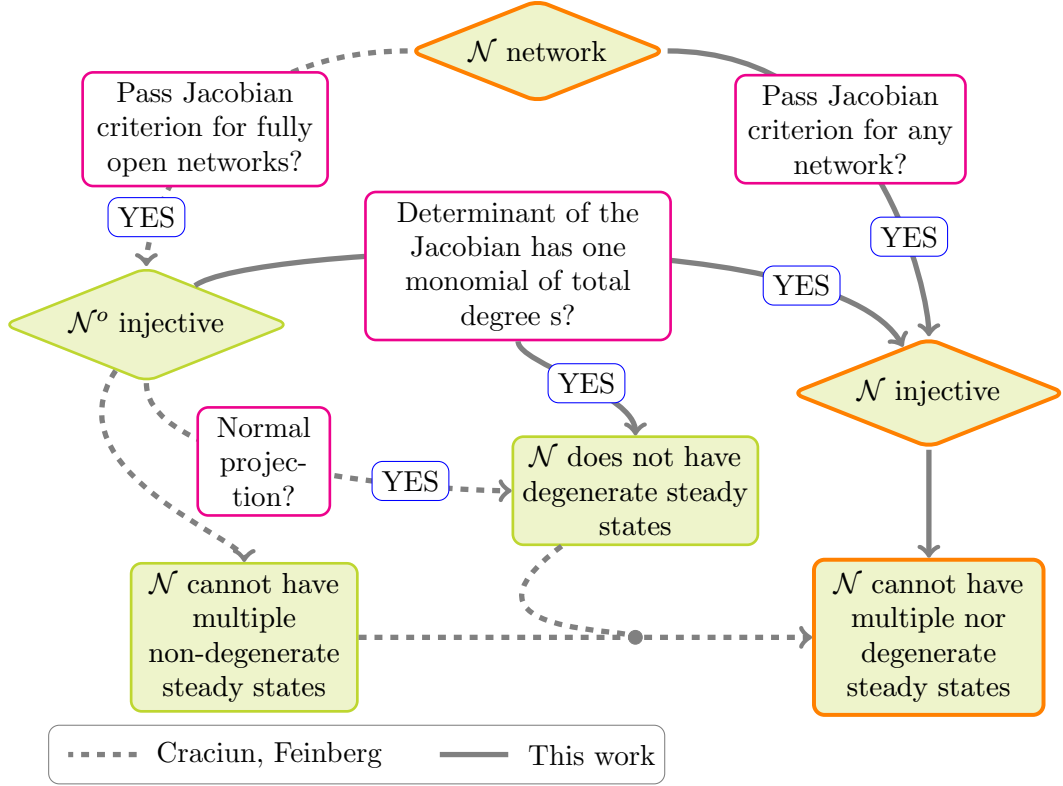

In this paper  we have provided a Jacobian criterion for the characterization of injective networks taken with mass-action kinetics. Injective networks have the important property that  
multiple positive steady states within any stoichiometry class cannot occur for any choice of rate constants. Further, the existence of  multiple boundary solutions of a certain type is also precluded. Importantly, if an injective network has a positive steady state, then it cannot have any other non-negative steady state.

Since injective networks are characterized by a non-singular Jacobian of the  species formation rate function (when restricted to the stoichiometric space), other interesting properties of this class of networks are expected. For instance, in \cite{shinar-flux} it is shown that their steady-state fluxes exhibit a certain degree of robustness  against variation in rate constants.

The main novelty of this work is that injectivity of a network can be assessed directly  avoiding the detour to fully open networks. In Figure \ref{summary} we show how our work relates to previous work on injectivity by Craciun and Feinberg in \cite{craciun-feinbergI,craciun-feinberg,craciun-feinberg-semiopen}.  

The Jacobian criterion presented here can be effectively implemented using any symbolic computation software like Mathematica. Suggested steps for its implementation are the following (using pseudo-Mathematica commands):

\begin{itemize}
\item Definitions:
\begin{quote}
\textbf{n}: number of species, \textbf{A}: stoichiometric matrix $\Gamma$,  \\
\textbf{c=c[1],...,c[n]}: Concentration vector, 
\textbf{v}: rate vector (with concentrations) $(k_{y\rightarrow y'} c^y)_{y\rightarrow y'\in\mmR}$, \\
 \textbf{F=A.v}: species formation rate function $f_{\kappa}$.
 \end{quote}
 \item Conservation laws and  associated extended rate function:
\begin{quote}
 \textbf{P=RowReduce[NullSpace[Transpose[A]]]}: find a reduced basis of $\Gamma^{\perp}$,  \\
 \textbf{ind}: vector of the indices of the first entry of each vector in \textbf{P}, \\
  \textbf{Ftilde}: remove entries \textbf{ind} of \textbf{F} and add the  entries \textbf{P.c}.
 \end{quote}
 \item Compute the determinant of the Jacobian of \textbf{Ftilde}:
 \begin{quote}
 \textbf{J=Table[\ D[\ Ftilde[[i]],\ c[j]\ ],\ \{i,1,n\},\{j,1,n\}]}, \\
 \textbf{D=Det[J]}.
  \end{quote}
 \item Determine the signs of the coefficients of \textbf{D}:
 \begin{quote}
 \textbf{m = MonomialList[D]}: monomials of \textbf{D}, \\
 \textbf{coeffs=DeleteDuplicates[m/.\{\textbf{k}$\rightarrow$1,\textbf{c}$\rightarrow$1\}]}: coefficient of each of the monomials (in  the rate constants and \textbf{c}), and delete duplicates, \\
 \textbf{Pos = Select[coeffs, Positive], Neg=Select[coeffs, Negative]}: select the positive and the negative coefficients, \\
If \textbf{Length[Pos]$>$0 and Length[Neg]=0}: $\mmN$ is injective, \\ 
If \textbf{Length[Pos]=0 and  Length[Neg]$>$0}: $\mmN$ is injective, \\
Otherwise: $\mmN$ is not injective.
   \end{quote}
\end{itemize}

In our experience, this procedure works fast and reliably for not-so-big networks (at least up to around 15-20 species). For bigger networks, the computational cost in finding the determinant might be too high for a standard computer. In that case, one  might  construct the lists \textbf{Pos} and \textbf{Neg} above  by computing one by one the signs of  $\det(\mmY(R)_{I})\det(\Gamma(R)_{I})$ for all sets $R$ of $s$ reactions and 
indices $I\in  \mathcal{O}_{d}(\mmN)$. If two conflicting (different) signs are found, the algorithm should  stop and the network is not injective. In Mathematica, the command \textbf{Subsets[list,\{s\}]} provides the subsets of a list containing exactly $s$ elements.

\section*{Acknowledgement}
EF is supported by a postdoctoral grant from the ``Ministerio de Educaci\'on'' of Spain and the project  MTM2009-14163-C02-01 from the ``Ministerio de Ciencia e Innovaci\'on''.  CW is supported by the Lundbeck Foundation, Denmark, The Danish Research Councils and the Leverhulme Trust, UK.  This work was done while EF and CW were visiting Imperial College London in fall 2011.


\end{document}